\documentclass[11pt]{article}
\usepackage{amsmath,amsthm,amssymb}
\usepackage{xifthen}
\usepackage{fullpage,bm,hyperref,cleveref,verbatim}
\usepackage[normalem]{ulem}
\usepackage{algorithmic}
\usepackage{algorithm}
\usepackage{caption}
\newtheorem{lem}{Lemma}

\newtheorem{thm}[lem]{Theorem}

\newtheorem{prop}[lem]{Proposition}
\newtheorem{cor}[lem]{Corollary}
\newtheorem{construct}[lem]{Construction}
\theoremstyle{definition}
\newtheorem{clm}{Claim}
\newtheorem{clmA}{Claim}

\newtheorem{conj}[lem]{Conjecture}
\newtheorem{defn}[lem]{Definition}

\newtheorem{myques}[lem]{Question}
\newtheorem*{mainthm}{Main Theorem}
\newtheorem*{planarthm}{Planar Theorem}
\newtheorem*{dthm}{$\bm{d}$-Degenerate Theorem}
\newtheorem{thmA}{Theorem}

\newtheorem{ques}[thmA]{Question}

\newcommand\floor[1]{\lfloor#1\rfloor}
\newcommand\ceil[1]{\lceil#1\rceil}
\renewcommand\geq\geqslant
\renewcommand\ge\geqslant
\renewcommand\leq\leqslant
\renewcommand\le\leqslant
\newcommand\vph\varphi

\def\aftermath{\par\vspace{-\belowdisplayskip}\vspace{-\parskip}\vspace{-\baselineskip}}

\newenvironment{clmproof}[1]{\par\noindent\underline{Proof.}\space#1}{\leavevmode\unskip\penalty9999\hbox{}\nobreak\hfill\quad\hbox{$\diamondsuit$}\smallskip}

\usepackage{tikz, xcolor, graphicx, subcaption}
\usetikzlibrary{bbox}
\usetikzlibrary{shapes}
\tikzstyle{uStyle}=[shape = circle, minimum size = 2pt, inner sep =2.0pt, outer sep = 0pt, draw, fill=white]
\tikzstyle{ubStyle}=[shape = circle, minimum size = 3pt, inner sep =1.25pt, outer sep = 0pt, draw, fill=white]
\tikzstyle{xStyle}=[shape = rectangle, minimum size = 2pt, inner sep =2.5pt, outer sep = 0pt, draw, fill=white]
\tikzstyle{yStyle}=[shape = diamond, minimum size = 2pt, inner sep =1.8pt, outer sep = 0pt, draw, fill=white]
\usetikzlibrary{decorations.markings, calligraphy}
\usetikzlibrary{decorations.pathmorphing}
\definecolor{mygray}{RGB}{192,192,192}
\definecolor{myothergray}{RGB}{128,128,128}
\colorlet{mygrayer}{mygray!80!white}
\colorlet{myothergrayer}{myothergray!75!white}

\newcommand{\stalactite}[3]{
\begin{scope}[yscale=.866]
\draw[thick] (#1,0) -- (#1+1,0) -- (#1+2,0) -- (#1+3,0) (#1,0) -- (#1+1.5,1) -- (#1+1,0) (#1+2,0) -- (#1+1.5,1) -- (#1+3,0) (#1,0) -- (#1
+0.5,-1) -- (#1+1,0) -- (#1+1.5,-1) -- (#1+2,0) -- (#1+2.5,-1) -- (#1+3,0);
\foreach \i/\j/\c in {#1+1.5/1/1, #1+0.5/-1/1, #1+1.5/-1/1, #1+2.5/-1/1, #1/0/#2, #1+2/0/#2, #1+1/0/#3, #1+3/0/#3}
\draw[thick] (\i,\j) node[ubStyle] {\tiny{\bf\c}};
\end{scope}
}

\def\mgray{gray!65!white}
\def\drawthick{\draw[line width=.8mm]}

\title{Equitably Coloring Planar and Outerplanar Graphs}
\author{Daniel W. Cranston\thanks{Department of Mathematics, William \& Mary, Williamsburg, VA, USA; \texttt{dcransto@gmail.com}}
\and Reem Mahmoud\thanks{Division of Science, NYU Abu Dhabi, Abu Dhabi, UAE; \texttt{rm7230@nyu.edu}}
}

\date{}
\begin{document}
\maketitle
\abstract{A proper $s$-coloring of an $n$-vertex graph is \emph{equitable} if every color class has size $\floor{n/s}$ or $\ceil{n/s}$.
A necessary condition to have an equitable $s$-coloring is that every vertex $v$ appears in an independent set of size at least  $\floor{n/s}$.
That is $\min_{v\in V(G)}\alpha_v\ge \floor{n/s}$.  Various authors showed that when $G$ is a tree and $s\ge 3$ this obvious necessary condition
is also sufficient.  Kierstead, Kostochka, and Xiang asked whether this result holds more generally for all outerplanar graphs.
We show that the answer is No when $s=3$, but that the answer is Yes when $s\ge 6$.  The case $s\in\{4,5\}$ remains open.
We also prove an analogous result for planar graphs, with a necessary and sufficient hypothesis.
Fix $s\ge 40$.  Let $G$ be a planar graph, and let $w_0,w_1$ be its $2$ vertices with largest degrees. 
If there exist disjoint independent sets $I_0, I_1$ such that $|I_0|=\floor{n/s}$ and $|I_1| = \floor{(n+1)/s}$ and $w_0,w_1\in I_0\cup I_1$, then
$G$ has an equitable $s$-coloring.  
}

\section{Introduction}
A proper coloring of an $n$-vertex graph with colors $1,\ldots,s$ is \emph{equitable} if each color class has size $\floor{n/s}$ or $\ceil{n/s}$.
Equitable coloring was introduced in 1968 by Gr\"{u}nbaum~\cite{grunbaum}.  
His paper studied a conjecture of Erd\H{o}s on partitioning cliques;
Gr\"{u}nbaum recast, and slightly generalized, this problem in the language of equitable colorings.  
A few years later Hajnal and Szemer\'{e}di~\cite{HS} confirmed this conjecture of Gr\"{u}nbaum, along with its precursor posed by Erd\H{o}s.

\begin{thmA}[\cite{HS}]
	\label{HS-thm}
	If $G$ is a graph with $\Delta(G)<s$, then $G$ has an equitable $s$-coloring.
\end{thmA}

Most subsequent work on equitable coloring fits into one of two regimes, both of which are discussed in the recent excellent survey~\cite{KKX-survey}
of Kierstead, Kostochka, and Xiang.  In the dense regime, $\Delta$ is fixed and most or all vertices can have 
degree close to $\Delta$.  The original proof~\cite{HS} of \Cref{HS-thm} was quite difficult; this difficulty motivated a series of 
simpler proofs~\cite{KK-short, MS, KKMS}, culminating with~\cite[Theorem~2]{KKX-survey}.  One nice extension of the Hajnal--Szemer\'{e}di Theorem is 
an ``Ore-degree'' analogue~\cite{KK-ore-type2}: the hypothesis ``$\Delta < s$'' is replaced by ``$d(v)+d(w)<2s$ whenever $vw\in E(G)$''.  
Another pretty development is a polynomial time algorithm~\cite{KKMS} 
to find the equitable $s$-coloring promised by the Hajnal--Szemer\'{e}di Theorem.  But the bulk of the work~\cite{CM, KK-ore-type, KK-4, KK-refinement, KKY, LW, nakprasit, YZ} in this area aims at the Chen--Lih--Wu
Conjecture~\cite{CLW}, which weakens the hypothesis to ``$\Delta\le s$'' but excludes graphs containing $K_{s,s}$ when $s$ is odd.

In contrast, the sparse regime focuses on graphs $G$ with average degree much lower than $\Delta$.  Here the aim is to prove that $G$ is equitably 
$s$-colorable for some $s$ much smaller than $\Delta$;
 for example, see~\cite{LSSY, WW}.  Early work on trees by Bollobas and Guy~\cite{bollobas-guy} showed that a 
tree $T$ is equitably $3$-colorable if $|T|\ge 3\Delta-8$ or $|T|=3\Delta-10$.  
For a graph $G$, and each $v\in V(G)$, let $\alpha_v(G)$ denote the size of a largest independent set containing $v$.  
For brevity, when the context is clear we simply write $\alpha_v$.  The result of Bollobas and Guy was extended as follows.

For a graph $G$ to have an equitable $s$-coloring, the condition $\alpha_v \ge \floor{n/s}$ is clearly necessary, 
since each color class has size at least $\floor{n/s}$.  (Throughout, we let $n:=|G|$.)
Chen and Lih~\cite{CL} and Miyata, Tokunaga, and Kaneko~\cite{MTK} showed that for all trees, and all $s\ge 3$, 
this obvious necessary condition is also sufficient; and Chang~\cite{chang} gave a more elegant proof.  

\begin{thmA}[Tree Theorem~\cite{chang,CL,MTK}]
    Fix an integer $s$ with $s\ge 3$.  If $G$ is a tree
    with $\alpha_v(G)\ge \floor{n/s}$ for all $v\in V(G)$, then $G$ has an equitable $s$-coloring.
\end{thmA}

A natural superclass of trees is outerplanar graphs, those with a plane
embedding where all vertices lie on the outer face.  Kostochka~\cite{kostochka-outerplanar} showed that every outerplanar graph $G$ with 
$\Delta\ge 3$ has an equitable $s$-coloring whenever $s\ge 1+\Delta/2$.  For some graphs this hypothesis on $s$ is needed, but for many it is not.  
In a vein similar to that of Bollobas and Guy, Pemmaraju (\cite{pemmaraju-manuscript}; see~\cite[Theorem~7]{pemmaraju-extends}) proved that every 
outerplanar graph $G$ has an equitable $6$-coloring when $n\ge 6\Delta$; indeed, his proof relies heavily on the result above of Bollobas and Guy. 

Kostochka, Nakprasit, and Pemmaraju~\cite{KNP,PNK-SODA} proved general results on equitable coloring of $d$-degenerate graphs.  
(Earlier related work includes~\cite{KN}.)
Since all planar graphs are $5$-degenerate, their work implies that every planar graph with $\Delta \le n/15$ is equitably $s$-colorable when $s\ge 80$.
Since all outerplanar graphs are
$2$-degenerate, this work of Pemmaraju---and our work in the present paper---can be viewed in their general framework. 
Finally Kierstead, Kostochka, and Xiang~\cite{KKX-survey} posed~Question~\ref{outerplanar-ques} below, 
asking whether the Tree Theorem can be extended to all outerplanar graphs.  
In this paper we largely answer their question. 

\begin{ques}[Kierstead,\,Kostochka,\,Xiang~\cite{KKX-survey}]
\label{outerplanar-ques}
\label{KKX-ques}
Does each outerplanar graph $G$ have an equitable $s$-coloring
if and only if every $v\in V(G)$ satisfies $\alpha_v\ge \floor{n/s}$ or, equivalently, 
$s\ge \ceil{\frac{n+1}{\alpha_v+1}}$?
\end{ques}

\noindent 
We mentioned above that the condition $\alpha_v\ge \floor{n/s}$ for all $v$ is clearly necessary for $G$ to have an equitable $s$-coloring.
The equivalence of this first condition with the second follows from starting with $s\ge (n+1)/(\alpha_v+1)$ and rewriting as $\alpha_v\ge (n+1-s)/s$.
Thus, Question~\ref{outerplanar-ques} above
asks whether we can extend the Tree Theorem to all outerplanar graphs.
Our Main Theorem shows that the answer to Question~\ref{outerplanar-ques} is Yes, mostly. 

\begin{mainthm}
Fix $s\ge 6$, and let $G$ be outerplanar.  If $\alpha_v\ge \floor{n/s}$ for all $v\in V(G)$, 
then $G$ has an equitable $s$-coloring.  The analogous statement is false for $s=3$ and open for $s\in\{4,5\}$.
\end{mainthm}

The short answer to Question~\ref{outerplanar-ques} is No, since when $s=3$ there exist counterexamples; see \Cref{s=3-lemma} below.  
In fact, the only $3$-colorings for these graphs $G$ have their largest color class of size $n/2$ and their next largest color class of size 
$n/4$, so they are far from equitably $3$-colorable.  (Indeed, these are the largest and smallest
values, respectively, that can occur for these sizes.) But the longer answer to Question~\ref{outerplanar-ques} 
is Yes, with the slightly stronger hypothesis $s\ge 6$.  For $s\in\{4,5\}$, the problem remains open, but we believe 
that the answer is also Yes.

\begin{conj}
	Fix $s\in \{4,5\}$.  If $G$ is outerplanar and $\min_{v\in V(G)}\alpha_v(G)\ge \floor{|G|/s}$, then $G$ has an equitable $s$-coloring.
	\label{45-conj}
\end{conj}

\noindent
As modest support for this conjecture, we show in the appendix that if it is true when $s=4$, then it is also true when $s=5$.
\smallskip

Our other main result in this paper is an analogue of our Main Theorem for all planar graphs.

\begin{planarthm}
	Fix $s\ge 40$, and let $G$ be a planar graph.  Let $w_0,w_1$ be its $2$ vertices with largest degrees, breaking ties arbitrarily.
	If there exist disjoint independent sets $I_0, I_1$ such that $|I_0|=\floor{n/s}$ and $|I_1| = \floor{(n+1)/s}$ and $w_0,w_1\in I_0\cup I_1$, 
	then $G$ has an equitable $s$-coloring.  
\end{planarthm}

In the Planar Theorem, we have a slightly stronger hypothesis.  The reason for this is that the hypothesis $\min_{v\in V(G)}\alpha_v\ge \floor{n/s}$ is
now insufficient to guarantee an equitable $s$-coloring.  Consider the graph $(K_2\vee P_{s^2})+sK_1$, formed from the join of $K_2$ and $P_{s^2}$
by adding $s$ isolated vertices.  If we denote by $w_1,w_2$ the $2$ high degree vertices, then we have $\alpha_{w_i}=s+1 = \floor{n/s}$.
And each other vertex $x$ has degree at most $4$, so $\alpha_x\ge \floor{n/s}$.  But this graph has no equitable $s$-coloring, since $w_1$ and $w_2$
must be colored distinctly, but each must be colored the same as all isolated vertices, in order to be in a sufficiently large color class.

The main appeal of the Planar Theorem, when compared with previous related work, is that its hypothesis is both necessary and sufficient.  
Indeed, we have removed the hypothesis on $\Delta$.  A secondary advantage is that it works for all $s\ge 40$.  We have made little effort 
to optimize the bound on $s$, since we believe that it is far from best possible.

\begin{conj}
	The Planar Theorem remains true if we replace ``40'' in its statement by ``5''.
	\label{planar-conj}
\end{conj}

If true, \Cref{planar-conj} is best possible.  That is, replacing ``40'' by ``4'' yields a statement that is false.  
It is straightforward\footnote{An \emph{extender} is a copy of the join $K_2\vee 2K_2$.  The single edge in the $K_2$ is the \emph{root edge} and the 
$2$ edges in the $2K_2$ are the \emph{leaf edges}.  Let $G_1$ be an extender.  For all $i\ge 2$, form $G_i$ from $G_{i-1}$ by picking an arbitrary
leaf edge $e$ in $G_{i-1}$ and doing the following.  Add an extender $H$ with its root edge identified with $e$.  Now add $2$ more extenders
with their root edges identified with the $2$ leaf edges of $H$.}, by identifying vertices in copies of $K_4$, to construct planar graphs $G_i$ 
with maximum degree $7$ and with order arbitrarily large in which every $4$-coloring has $2$ color classes of size $|G_i|/3$ and $2$ color classes 
of size $|G_i|/6$.

The proof of the Planar Theorem is relatively short, so we present it in Section~\ref{planar-sec}.
	We begin \Cref{outerplanar-sec} with the construction in our Main Theorem for $s=3$.  In \Cref{proof-overview-sec}, we outline the proof of the 
	upper bound in the Main Theorem.  In \Cref{s=8-sec} we prove the Main Theorem when $s\ge 8$, subject to a key ``Partitioning Lemma''.
	In \Cref{partitioning-sec} we prove the Partitioning Lemma, and in \Cref{s=6-sec} we finish the proof of the Main Theorem, when $s\ge 6$.
	Finally, in the appendix we offer some evidence supporting Conjecture~\ref{45-conj}; we also provide some details that we previously deferred,
	in the proof of the Partitioning Lemma.

\section{Planar Graphs}
\label{planar-sec}
We need the following result~\cite{ELM} of Esperet, Lemoine, and Maffray. (We write $[k]$ for $\{1,\ldots,k\}$.)
\begin{thmA}[\cite{ELM}]
	\label{ELM-thm}
	Every planar graph $G$ has a partition of $V(G)$ as $F_1\uplus F_2\uplus F_3\uplus F_4$ such that each $F_i$ induces a forest and 
	$||F_i|-|F_j||\le 1$ for all $i,j\in[4]$.
\end{thmA}
	The proof of \Cref{ELM-thm} is short and elegant, and it hinges crucially on Borodin's result~\cite{borodin-5acyclic} that every planar graph has 
	acyclic chromatic number at most $5$.  An easy application of \Cref{ELM-thm} implies the following lemma.
\begin{lem}
	Fix a positive integer $s$. If $G$ is planar and $\Delta(G)\le \frac{s-2}{4s}|G|$, then $G$ has an equitable $4s$-coloring.
	In particular, if $\Delta(G)\le |G|/5$, then $G$ has an equitable $40$-coloring.
	\label{planar-lem}
\end{lem}
\begin{proof}
	The second statement follows from the first by letting $s:=10$.
	By the previous theorem, we partition $V(G)$ as $F_1\uplus F_2\uplus F_3\uplus F_4$ where each $F_i$ induces a forest and $||F_i|-|F_j||\le 1$
	for all $i,j\in[4]$.  By the Tree Theorem\footnote{The Tree Theorem, as stated, applies only to trees; but we will often
want to apply it to forests.  So we now prove the equivalence.  Suppose that $F$ is a forest with $\min_{v\in V(F)}\alpha_v(F)\ge\floor{n/s}$ for 
some integer $s$.  Let $T_1$ and $T_2$ be trees of $F$, and let $v_1$ (resp.~$v_2$) lie in the weakly smaller color class of a $2$-coloring of $T_1$ 
(resp.~$T_2$).  Let $F':=F+v_1v_2$.  It is easy to check that also $\min_{v\in V(F')}\alpha_v(F')\ge \floor{n/s}$.  By induction on the number of trees in
$F$, we reduce to the case that $F$ is a single tree.  (Chang~\cite{chang} was the first to state this more general version, allowing forests.)}, we equitably $s$-color each $F_i$.  To see that the theorem applies, we note for each $i\in [4]$
	that $\Delta(F_i)\le \Delta(G)\le \frac{s-2}{4s}|G|\le \frac{s-2}s|F_i|+1$.  So for each $v\in F_i$ we have
	$$\alpha_v(F_i) \ge 1+\left\lceil\frac{|F_i|-d_{F_i}(v)-1}2\right\rceil\ge1 + \left\lceil\frac{|F_i| - \frac{s-2}{s}|F_i| - 2}2\right\rceil 
	= \ceil{|F_i|/s}.
	$$
	Using colors $si,si-1,\ldots,s(i-1)+1$ on $F_i$, for each $i\in[4]$, gives the equitable $4s$-coloring.
\end{proof}

\Cref{planar-lem} strengthens the analogous result of Kostochka, Nakprasit, and Pemmaraju~\cite{KNP} that we mentioned above; our result 
allows larger maximum degree or requires fewer colors (or some of both).  
However, we prefer to replace the bound on $\Delta$ with some hypothesis that is necessary for the desired equitable coloring to exist.  
\begin{planarthm}
	Fix $s\ge 40$, and let $G$ be a planar graph.  Let $w_0,w_1$ be its $2$ vertices with largest degrees, breaking ties arbitrarily.
	If there exist disjoint independent sets $I_0, I_1$ such that $|I_0|=\floor{n/s}$ and $|I_1| = \floor{(n+1)/s}$ and $w_0,w_1\in I_0\cup I_1$, 
	then $G$ has an equitable $s$-coloring.  
\end{planarthm}
\begin{proof}
	Let $G_0:=G$, let $G_1:=G_0-I_0$, and let $G_2:=G_1-I_1$.
	For each integer $j\ge 0$, we do the following.  If $4$ divides $s-j$ and $\Delta(G_j)$ is sufficiently small, then by \Cref{planar-lem} we 
	partition $V(G_j)$ into $4$ forests and equitably color each forest with $(s-j)/4$ colors.  But if $\Delta(G_j)$ is too big, or $4$ does not 
	divide $s-j$, then we delete an independent set $I_j$ (of the right size) containing a vertex $w_j$ of maximum degree in $G_j$. We let $j:=j+1$, 
	call our new graph $G_j$, and repeat the loop.  All that remains is to verify that this procedure indeed produces an equitable $s$-coloring.

	We show that the algorithm above finishes at some point with $j\le s-12$.
	Because $G_2$ is planar, it is $4$-colorable.  If $\Delta(G_2)\le 2n/3$, then $\alpha_{w_2}(G_2)\ge 1 + (38n/40 - 2n/3-1)/4\ge n/15$.
	Otherwise, $|N_G(w_0)\cap N_G(w_1)|\ge 2(2n/3) - n = n/3$; since $G$ has no $K_{3,3}$, at most $2$ of these vertices are in $N_G(w_2)$.
	So $\alpha_{w_2}(G_2) \ge 1 + (n/3-2-2n/40)/4 \ge n/15$.  Similarly, we can find $I_3$ in $G_3$ containing $w_3$; here we condition on 
	whether or not $\Delta(G_3)\le n/2$.  

	Now we assume that $j\ge 4$.
	For each $j$, let $n_j:=|G_j|$.  To prove correctness, we maintain the invariant, when $4\le j\le s-12$, that $\Delta(G_j) \le 2n_j/3$.
	If not, then 
	$$|E(G)| \ge \sum_{i=0}^jd_{G_i}(w_i) \ge (j+1)d_{G_j}(w_j) \ge (j+1)\frac23n_j\ge (j+1)\frac23\cdot\frac{s-j}sn \ge 3n,$$
	contradicting the fact that $G$ has fewer than $3n$ edges.  The final inequality $(j+1)(s-j)\ge 9s/2$ holds because the left side 
	is concave in $j$ and the inequality does indeed hold at the endpoints, $4$ and $s-12$.  This bound on $\Delta(G_j)$ ensures that always 
	$G_j$ will have
	an independent set $I_j$ containing $w_j$ of size at least $1+\ceil{(n_j - (d_{G_j}(w_j)+1))/4}\ge \ceil{n_j/12}\ge \ceil{(n(s-j)/s)/12}
	\ge\ceil{n/s}\ge \floor{(n+j)/s}$, as desired.  Thus, there exists the independent set $I_j$ containing $w_j$.

	Now we must show that the algorithm halts with $j\le s-12$.  For brevity, let $H_j:=G_{j+s-40}$, let $\tilde{n}_j:=|H_j|$, and let 
	$\tilde{w}_j:=w_{j+s-40}$.  We just focus on the run of the algorithm starting with $H_0$ and show that if it does not halt by
	the time that we reach $H_{28}$, then the sum of degrees of $\tilde{w}_0,\ldots,\tilde{w}_{28}$, at the times they are deleted, 
	exceeds the number of edges in $H_0$, a contradiction.
	Intuitively, we are just trying to construct the final $40$ color classes in an equitable $s$-coloring.

	If we cannot equitably $40$-color $H_0$ by \Cref{planar-lem}, then this is because $d_{H_0}(\tilde{w}_0)\ge \frac{8}{40}\tilde{n}_0$.
	Similarly, if we cannot equitably $(40-4j)$-color $H_{4j}$, then $d_{H_{4j}}(\tilde{w}_{4j})\ge
	\frac{8-j}{4(10-j)}\tilde{n}_{4j}\ge \frac{8-j}{4(10-j)}\tilde{n}_{0}\frac{40-4j}{40} = \frac{8-j}{40}\tilde{n}_0$.
	Note that $d_{H_{4j-i}}(\tilde{w}_{4j-i})\ge d_{H_{4j}}(\tilde{w}_{4j})$ for all $j\in[7]$ and all $i\in [3]$.
	Thus, $|E(H_0)| \ge \sum_{j=0}^{28}d_{H_j}(\tilde{w}_j)\ge
	d_{H_0}(\tilde{w}_0) + \sum_{j=1}^{7}4d_{H_{4j}}(\tilde{w}_{4j}) \ge \tilde{n}_0\left(\frac8{40}+\sum_{j=1}^7\frac{8-j}{10}\right) 
	= \frac1{10}\tilde{n}_0\left(2 + 8(7)/2\right) = 3\tilde{n}_0$.
        So $H_0$ has at least $3\tilde{n}_0$ edges, which is too many since $H_0$ is planar with order~$\tilde{n}_0$. 
\end{proof}

It is an open question, posed in~\cite{ELM} and repeated in~\cite{KKX-survey,KOZ}, whether every planar graph has an equitable partition into $3$ forests.
If the answer is Yes, then we can strengthen \Cref{planar-lem}: If $G$ is planar and $\Delta(G)\le \frac{s-2}{3s}|G|$, then $G$ has an equitable $3s$-coloring.
Using this stronger lemma, we can repeat the proof of the Planar Theorem, but now weaken the hypothesis to $s\ge 30$.

Reusing the ideas in the previous proof, we can prove the following more general statement.

\begin{dthm}
	Fix $d$ a positive integer and $s\ge (2d+5)3^{d-1}$.  Let $G$ be $d$-degenerate.  Let $w_0,\ldots,w_{d-1}$ be the $d$ vertices of largest 
	degrees, breaking ties arbitrarily.  If there exist disjoint independent sets $I_0,\ldots,I_{d-1}$ in $G$ with $|I_j| = \floor{(n+j)/s}$ for
	all $j\in\{0,\ldots,d-1\}$ and $w_0,\ldots,w_d\in\cup_{j=0}^{d-1}I_j$, then $G$ has an equitable $s$-coloring.
\end{dthm}
\begin{proof}[Proof Sketch]
	We omit a formal proof, since all the important ideas appear in the previous proof.  We just comment on a few points.  Every $d$-degenerate
	graph has at most $dn$ edges.  By a result in~\cite{KNP}, every $d$-degenerate graph can be equitably partitioned into $3^{d-1}$ forests.
	To equitably $s$-color $G$, we again repeatedly delete an independent set (of the right size) containing a vertex of maximum degree.  
	When we have only $(2d+5)3^{d-1}$
	colors remaining, we try to partition into forests and finish by the Tree Theorem.  If we cannot, then we resume deleting independent sets
	(containing vertices of maximum degree) and try again with forests each time the number of remaining colors is a multiple of $3^{d-1}$.  If we 
	repeatedly fail with forests, then we eventually delete more than $d|H_0|$ edges from $H_0$, where $H_0$ is the remaining subgraph when 
	the number of remaining colors is $(2d+5)3^{d-1}$.  Recall that, by hypothesis, we have independent sets 
	$I_0,\ldots, I_{d-1}$.  Finding $I_d$ containing $w_d$ must be handled specially, akin to $I_2$ and $I_3$ in the previous proof.
	Thereafter, we maintain the invariant $\Delta(G_j)\le \frac{d}{d+1}|G_j|$.  This guarantees that $G_j$ has independent set 
	$I_j$ containing a vertex $w_j$ of maximum degree in $G_j$.
\end{proof}

The hypothesis on $I_0,\ldots,I_{d-1}$ in the previous theorem is necessary, as witnessed by the $d$-degenerate graph $K_d\vee (s^2+2s)K_1+(ds+d^2)K_1$.  
We can put any $d-1$ of its vertices with largest degrees into (correctly sized) color classes for an equitable $s$-coloring, but not all $d$ vertices 
with largest degrees.  Thus, the graph has no equitable $s$-coloring.
The main drawback of the previous theorem is the strong hypothesis on $s$.  This prompts the following question.

\begin{myques}
	In the previous theorem, can we weaken the hypothesis on $s$ to ``$s\ge f(d)$'' where $f(d)$ is polynomial in $d$, perhaps even linear?
\end{myques}

\section{Outerplanar Graphs}
\label{outerplanar-sec}
In this section, we begin our proof of the Main Theorem.  We start with the construction showing that the Main Theorem becomes false if we let $s:=3$.

\begin{construct}
\label{s=3-lemma}
	There is an infinite class $\mathcal{G}$ of outerplanar graphs such that, for every $G\in\mathcal{G}$, (i) $G$ is a near-triangulation, (ii) $\Delta(G)=5$, (iii) $\min_{v\in V(G)}\alpha_v \ge \floor{|G|/3}$, and (iv) $G$ has no equitable 3-coloring.
\end{construct}

\begin{proof}

A \textit{stalactite} is formed as follows; see Figure~\ref{counterexample1}. Starting from $K_{1,4}$, with leaves $x_1,\dots,x_4$ from left to right, for each $i\in[3]$ add edge $x_ix_{i+1}$ and add a new vertex $y_i$ and edges $y_ix_i$ and $y_ix_{i+1}$. 
Let $\varphi$ be an arbitrary coloring with colors in $[k]$. For each $j\in[k]$, denote by $\varphi^{-1}(j)$ the set of vertices colored $j$.

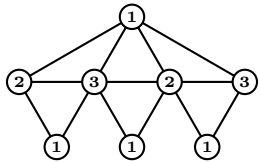
\begin{figure}[!h]
\centering
\begin{tikzpicture}[scale=1]%
\stalactite{-2}{2}{3}    
\end{tikzpicture}
\caption{A stalactite; equivalently, $G_1$ from the construction in Lemma~\ref{s=3-lemma}.}
\label{counterexample1}
\end{figure}

	Let $G_1$ be a stalactite. Clearly, $\Delta(G_1)=5$. Since $G_1$ is a near-triangulation, it has a unique $3$-coloring up to permuting colors. 
	Let $\varphi_1$ be the $3$-coloring of $G_1$ shown in Figure~\ref{counterexample1}. Since $|\varphi_1^{-1}(i)|\ge2=\lfloor{|G_1|}/3\rfloor$ for each 
	$i\in[3]$, we have $\min_{v\in V(G_1)}\alpha_v\ge\floor{|G_1|/3}$. Moreover, 
	$|\varphi_1^{-1}(1)|=4={|G_1|}/{2}$ and $|\varphi_1^{-1}(2)|=|\varphi_1^{-1}(3)|=
	2={|G_1|}/{4}$. So $\varphi_1$ is not equitable. Hence, $G_1$ satisfies (i)--(iv).

For each $i\ge2$, form $G_i$ from $G_{i-1}$ by merging two disjoint stalactites $H_1$ and $H_2$ with $G_{i-1}$ as follows; see Figure~\ref{counterexample2}. 
To merge $H_1$ with $G_{i-1}$, introduce a new disjoint edge $v_1w_1$. Join $v_1$ (resp.~$w_1$) to the rightmost vertex (resp.~two rightmost vertices) 
of $H_1$; similarly, join $v_1$ (resp.~$w_1$) to the leftmost vertex (resp.~two leftmost vertices) of $G_{i-1}$.
Analogously, merge $H_2$ and $G_{i-1}$ by introducing edge $v_2w_2$, with endpoints adjacent to the rightmost vertices in $G_{i-1}$ and adjacent to
the leftmost vertices in $H_2$.
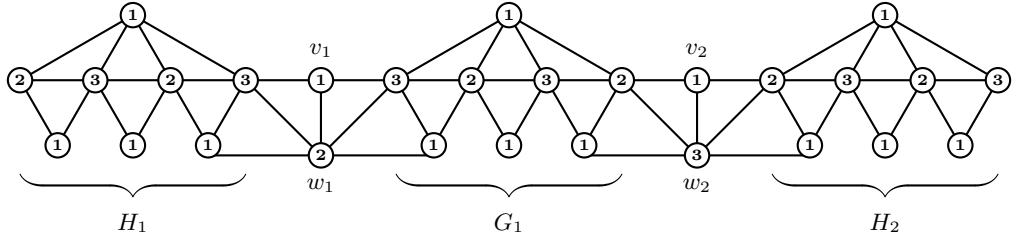
\begin{figure}[!t]
\centering
\begin{tikzpicture}[scale=1]
\draw[thick] (2,0) -- (3,0) -- (4,0) (2,0) -- (3,-1) -- (4,0) (1.5,-1) -- (3,-1) -- (4.5,-1) (7,0) -- (8,0) -- (9,0) (7,0) -- (8,-1) -- (9,0) (6.5,-1) -- (8,-1) -- (9.5,-1); 
\stalactite{-1}{2}{3} 
\stalactite{4}{3}{2}
\stalactite{9}{2}{3}
	\draw[thick] (3,0) node[ubStyle] {\tiny{\bf1}} -- (3,-1) node[ubStyle] {\tiny{\bf2}} (8,0) node[ubStyle] {\tiny{\bf1}} -- (8,-1) node[ubStyle] {\tiny{\bf3}} (3,0.4) node {\footnotesize{$v_1$}} (3,-1.4) node {\footnotesize{$w_1$}} (8,0.4) node {\footnotesize{$v_2$}} (8,-1.4) node {\footnotesize{$w_2$}} (0.5,-1.9) node {\footnotesize{$H_1$}} (5.5,-1.9) node {\footnotesize{$G_1$}} (10.5,-1.9) node {\footnotesize{$H_2$}};
\draw[thick, decorate, decoration={calligraphic brace, mirror, amplitude=3mm}] (-1,-1.25) -- (2,-1.25);
\draw[thick, decorate, decoration={calligraphic brace, mirror, amplitude=3mm}] (4,-1.25) -- (7,-1.25);
\draw[thick, decorate, decoration={calligraphic brace, mirror, amplitude=3mm}] (9,-1.25) -- (12,-1.25);
\end{tikzpicture}
\caption{Constructing $G_2$ from $G_1$ in the proof of Lemma~\ref{s=3-lemma}.}
\label{counterexample2}
\end{figure}

	Clearly, $\Delta(G_i)=5$ and $G_i$ is a near-triangulation. Note that $|G_i|=|G_{i-1}|+20$. Let $\varphi_{i-1}$ be the unique 3-coloring (up to 
	renaming colors) of $G_{i-1}$ that uses color $1$ most often. Now $\varphi_{i-1}$ extends to a unique 3-coloring $\varphi_i$ of $G_i$; 
	see Figure~\ref{counterexample2}. So color class 1 grows by 10, and color classes 2 and 3 each grow by 5.  That is, 
	$|\varphi_i^{-1}(1)|=|\varphi_{i-1}^{-1}(1)|+10={|G_i|}/{2}$ and 
	$|\varphi_i^{-1}(j)|=|\varphi_{i-1}^{-1}(j)|+5={|G_i|}/{4}$ for each $j\in\{2,3\}$. Thus, $\varphi_i$ is not equitable.

	We now show that $\min_{x\in V(G_i)}\alpha_x\ge\floor{|G_i|/3}$. Note that $|\varphi_i^{-1}(1)|={|G_i|}/{2}>\lfloor{|G_i|}/{3}\rfloor$. 
	Moreover, for each vertex $x\in\varphi_i^{-1}(2)\cup \varphi_i^{-1}(3)$, we have $|N(x)\cap \varphi_i^{-1}(1)|\le3$. 
	Let $S:=\{x\}\cup\varphi_i^{-1}(1)\setminus N(x)$. 
	Note that $S$ is an independent set and $|S|\ge1+{|G_i|}/{2}-3\ge\floor{{|G_i|}/{3}}$ since $|G_i|\ge 12$ whenever $i\ge2$. 
	Thus, $\min_{x\in V(G_i)}\alpha_x\ge\floor{|G_i|/3}$. 
\end{proof}
\subsection{Proof Overview}
\label{proof-overview-sec}
The proof of the Main Theorem is somewhat involved.  So here we first present an overview.
We typically refer to an ``equitable $s$-coloring''.  Our goal is to prove the following statement, for all $s\ge 6$: If $G$ is an outerplanar graph 
with $\alpha_v(G)\ge \floor{n/s}$ for all $v\in V(G)$, then $G$ has an equitable $s$-coloring.  The first step is to prove that if the statement holds for 
$s-1$, then it also holds for $s$.  This is straightforward to show for $s\ge 6$, as we do in \Cref{reduce-to-small-s-lem}.  The reduction also holds for 
$s=5$, but the proof is harder, so we defer it to the appendix since it is not needed to prove the Main Theorem.  
Thus, the focus of the paper is proving the statement when $s=6$.

A \emph{forest equipartition} for $G$ is a partition $F_1\uplus F_2$ of $V(G)$ such that $||F_1|-|F_2||\le 1$ and $G[F_i]$ is a forest, for each 
$i\in[2]$.  For brevity, when the meaning is clear, we write $F_i$ both for the vertex set and for the subgraph it induces.
Pemmaraju (in an unpublished manuscript~\cite{pemmaraju-manuscript}; see Theorem~9 of~\cite{pemmaraju-extends}) showed that if $G$ is outerplanar 
with $\Delta(G)\le n/6$, then $G$ has a forest equipartiton $F_1\uplus F_2$.  His idea was as follows.  
By the Tree Theorem, we 
equitably $3$-color $F_1$ with colors $1,2,3$; similarly, we equitably $3$-color $F_2$ with colors $4,5,6$.  Together these give an equitable
$6$-coloring of $G$.  Pemmaraju's manuscript is not widely available, and we have not seen it; but his equipartitioning result is easy to recreate, as
we describe below.  In fact, we significantly strengthen it, by allowing a much larger maximum degree~of~$G$.

Why did Pemmaraju require that $\Delta(G)\le n/6$? Because this implies $\Delta(F_i)\le \Delta(G)\le n/6\le (|F_i|+1)/3$.
Now, since $F_i$ is $2$-colorable, we have $\alpha_v(F_i)\ge 1+(|F_i|-(d_{F_i}(v)+1))/2 = (|F_i|+1-d_{F_i}(v))/2 \ge (|F_i|+1)/3$.
So $F_i$ is indeed equitably $3$-colorable by the Tree Theorem.  

We strengthen the above equipartitioning result to allow $\Delta(G)\le 2\ceil{n/6}+3$;
that is, the maximum degree can now be more than double what was allowed previously.  To do so, we ensure that for each $v\in F_i$ the forest $F_i$
contains at most half of the edges incident with $v$ in $G$.  To get better control on $d_{F_i}(v)$ we want $G$ to be a maximal outerplanar graph.
This maximality does not improve our degree bounds in $G$, but it does tell us the precise structure of $G$ at certain ``ends'' (leaves in the weak 
planar dual) that we use to proceed by induction.  

So we first add edges, but no vertices, to $G$ until we reach a maximal outerplanar graph $G'$.  The key is to do this 
while maintaining that $\Delta(G')\le 2\ceil{n/6}+3$.  We can nearly do so (at most $2$ degrees may increase to $2\ceil{n/6}+4$), but the proof is 
detailed.  Once we have $G'$, we show how to partition $V(G')$ as $F_1\uplus F_2$ so that $d_{F_i}(v)\le\ceil{n/6}+1$ for each
$i\in [2]$ and for all $v\in F_i$.  We call this result the \emph{Partitioning Lemma}; see~\Cref{partitioning-lem} below.
(After proving the Partitioning Lemma, we sketch a short proof of Pemmaraju's weaker variant.)

As a warmup, we initially prove the Main Theorem when $s\ge 8$, assuming the Partitioning Lemma.
Then, after proving the Partitioning Lemma, we use it to prove the Main Theorem when $s\ge 6$.
We call a vertex $w$ in a graph $G$ \emph{dangerous} if $d_G(w)\ge 2\ceil{n/6}+4$.  A straightforward counting argument (see \Cref{few-big-lem})
shows that each outerplanar graph has at most $2$ dangerous vertices.  It is easy to verify that the Partitioning Lemma and the Tree Theorem combine 
to handle the case when the number of dangerous vertices in $G$ is $0$.  
So we need only consider the cases that this number is $2$ or $1$.

Suppose that $G$ has exactly $2$ dangerous vertices, say $w_1$ and $w_2$ (\Cref{two-dangerous-sec}).
Now we find $2$ independent sets, say $I_1$ and $I_2$, in $N[w_1]\cup N[w_2]$ and with sizes $\ceil{n/6}$ and $\ceil{(n-1)/6}$ such that 
$w_1\in I_1$ and $w_2\in I_2$; we color $I_1$ and $I_2$ with, respectively, colors $5$ and $6$.  We $3$-color, with colors $1,2,3$, a subgraph $G'$ 
of $G$ with order about $n/3$ that contains all of $G\setminus(N[w_1]\cup N[w_2])$, and finally we extend this partial $5$-coloring to an equitable 
$6$-coloring.

Suppose instead that $G$ has exactly $1$ dangerous vertex, $w$ (\Cref{one-dangerous-sec}).  Now we show that one of two strategies works.
Either we explicitly construct the equitable $6$-coloring, or else we again find a forest equipartition, being careful to ensure that
the forest $F_1$ containing $w$ still satisfies $\alpha_{F_1}(w)\ge \floor{|F_1|/3}$.  Thus, we can again finish by $2$ applications of the Tree Theorem.

\subsection{Proving the Main Theorem when \texorpdfstring{$s\ge 8$}{s >= 8}}
\label{s=8-sec}
Here we prove the Main Theorem when $s\ge 8$.  
We often implicitly use the following easy observation.

\begin{prop}
\label{linear-forest}
If $G$ is outerplanar, then $G[N(v)]$ is a disjoint union of paths, for all $v\in V(G)$.
\end{prop}
\begin{proof}
If $G[N(v)]$ contains a cycle, then $G[N[v]]$ contains a $K_4$-minor, contradicting that $G$ is outerplanar.  So it suffices to show that $G[N(v)]$ has maximum degree at most $2$.
Suppose, to the contrary, that there exist $w,x_1,x_2,x_3\in N(v)$ with $x_1,x_2,x_3 \in N(w)$.  Now $G[\{v,w,x_1,x_2,x_3\}]$ contains $K_{2,3}$, again contradicting that $G$ is outerplanar.
\end{proof}
	
To prove the Main Theorem for all $s$ at least $6$ (resp.~$8$), we first reduce to the case $s=6$ (resp.~$s=8$).  
This reduction is handled by the following lemma.
\begin{lem}
\label{reduce-to-small-s-lem}
\label{lem1}
	Fix an integer $s$ with $s\ge 6$. If $G'$ has an equitable $(s-1)$-coloring whenever $G'$ is outerplanar and 
	$\min_{w\in V(G')}\alpha_{w}(G')\ge \floor{|G'|/(s-1)}$,
	then $G$ has an equitable $s$-coloring whenever $G$ is outerplanar and $\min_{x\in V(G)} \alpha_x(G)\ge \floor{|G|/s}$.
	Thus, for each integer $s_0$ with $s_0\ge 5$, to prove the Main Theorem whenever $s\ge s_0$, it suffices to prove it when $s=s_0$.
\label{reduc-to-small-lem}
\end{lem}
\begin{proof}
To begin, we prove the second statement, assuming the first.
We use induction on $s-s_0$; the base case, $s-s_0 = 0$, holds by assumption.
The induction step, when $s -s_0\ge 1$, holds by the first statement.  

Now we prove the first statement.
Let $v$ be a vertex of maximum degree.
Let $I_v$ be an independent set containing $v$ of size $\lfloor n/s\rfloor$; such a set exists by hypothesis.
Let $G':=G-I_v$ and let $n':=|G'|$.  For brevity, we let $\alpha'_{w}:=\alpha_{w}(G')$.
Recall, from discussion following~Question~\ref{KKX-ques}, that the condition $\alpha_w\ge \floor{n/s}$ is equivalent to the condition $s\ge\ceil{(n+1)/(\alpha_w+1)}$.
For every $w\in V(G')$ we get $\lceil\frac{n'+1}{\alpha'_w+1}\rceil \le 
\lceil\frac{n-\lfloor n/s\rfloor+1}{\lceil(n-\lfloor n/s\rfloor -d(w)-1)/3\rceil+2}\rceil\le \lceil\frac{n-n/s+2}{(n-n/s -d(w)-1)/3+2}\rceil$. 
We want to show, for every $w\in V(G')$, that this final term is at most $s-1$.  
So assume not. 

Now by (tedious, but straightforward) algebra, provided below, we 
get $d(w) > (n(s-5+4/s)+{5s-11})/(s-1) = n(s-5+4/s)/(s-1)+5-6/(s-1)$.  
We next show that $d(w) \ge n/3+4$.
If $s\ge 7$, then $d(w) > n(7-5)/(7-1)+4=n/3+4$. 
And if $s=6$, then $d(w) > n(6-5+4/6)/(6-1) + 19/5 = n/3+19/5$; but since $d(w)$ is an integer, this implies $d(w)\ge n/3+4$.  
In both cases $d(w) \ge n/3+4$, as desired.
But recall that $d(v) =\Delta (G) \ge d(w)$.
Note that $N_G(v)\subseteq V(G')$.
We would like to find a big independent set in $G'$ that contains $w$; so we look in $\{w\}\cup N(v)$.  Since $G$ is outerplanar, it contains no copy of $K_{2,3}$; thus $|N(v)\cap N[w]|\le 3$.  Hence, by \Cref{linear-forest} vertex
$w$ appears in an independent set (in $G'$) of size at least $1+(d(v)-3)/2$; that is, $\alpha'_w\ge 1 + (d(v)-3)/2 = (d(v)-1)/2 \ge (n/3+3)/2 = n/6+3/2$.  If $s\ge 7$, then
$\lceil\frac{n'+1}{\alpha'_w+1}\rceil\le \lceil\frac{n}{n/6}\rceil = 6$, as desired.
And if $s = 6$, then $n' = n-\lfloor n/6\rfloor \le 5n/6+1$.
So
$\lceil\frac{n'+1}{\alpha_w'+1}\rceil\le \lceil\frac{5n/6+2}{n/6+5/2}\rceil \le 5$, as desired.
Thus, by hypothesis, $G'$ has an equitable $(s-1)$-coloring.  
Combining this with a new color on $I_v$ gives an equitable $s$-coloring of $G$.

Below are the details of the ``tedious, but straightforward'' algebra mentioned above. 
\begin{align*}
\frac{n-n/s+2}{(n-n/s-d(w)-1)/3+2} &> s-1\\
n - n/s+2 &> 2s-2 + \frac13(ns-n-n+n/s-sd(w)+d(w)-s+1)\\
3n-3n/s+6 &> 6s-6 + ns -2n + n/s -sd(w) + d(w)-s+1\\
(s-1)d(w) &> n(s-5+4/s)+ {~5s-11}\\
d(w) &> \frac{n(s-5+4/s)+{5s-11}}{s-1}
\end{align*}
\aftermath
\end{proof}

	To prove \Cref{main-thm-8}, we need the Partitioning Lemma, below.  We defer its proof to \Cref{partitioning-sec}.

\begin{lem}[Partitioning Lemma]
	If $G$ is outerplanar, then $G$ has a forest equipartition $F_1\uplus F_2$, where 
	$d_{F_i}(v)\le \max\{\ceil{n/6}+1,\floor{d_G(v)/2}\}$ for all $i\in[2]$ and all $v\in F_i$.
\end{lem}

\begin{thm}
	If $G$ is outerplanar with $\min_{v\in V(G)} \alpha_v \ge \floor{n/8}$, then $G$ has an equitable $8$-coloring.  So for each $s\ge 8$,
	if $G$ is outerplanar with $\min_{v\in V(G)} \alpha_v \ge \floor{n/s}$, then $G$ has an equitable $s$-coloring.
	\label{main-thm-8}
\end{thm}
\begin{proof}
	The second statement follows from the first by \Cref{reduc-to-small-lem}.  Now we prove the first.

	Suppose that $\Delta(G)\le \ceil{n/2}+2$.  
	By the Partitioning Lemma, we have an equipartition into forests $F_1$ and $F_2$ with $||F_1|-|F_2||\le 1$
	and $d_{F_i}(v)\le \ceil{n/4}+1$ for all $i\in[2]$ and all $v\in F_i$.
	Now $\alpha_v(F_i)\ge 1 + \ceil{(|F_i|-(\Delta(F_i)+1))/2}\ge \ceil{(|F_i|-\ceil{n/4})/2}\ge \floor{|F_i|/4}$.
	To verify the final inequality, let $k:=\floor{n/8}$ and let $\ell:=n-8k$.  When $\ell\in\{0,1,2,3\}$, we have
	$\ceil{(|F_i|-\ceil{n/4})/2}\ge \ceil{(4k-(2k+1))/2}=\ceil{(2k-1)/2}=k\ge \floor{|F_i|/4}$.  And when $\ell\in\{4,5,6,7\}$,
	we have $\ceil{(|F_i|-\ceil{n/4})/2}\ge \ceil{(4k+2-(2k+2))/2}=k\ge \floor{|F_i|/4}$.  
	By the Tree Theorem, each forest $F_i$ has a $4$-coloring with each color class of size $k$ or $k+1$.  Together, these $4$-colorings
	give an equitable $8$-coloring of $G$.

	Instead assume that $\Delta(G)\ge \ceil{n/2}+3$.
	Now let $v$ be a vertex of maximum degree, and let $I_v$ be an independent set containing $v$ of size $\floor{n/8}$.
	Let $G':=G-I_v$.  Let $w$ be a vertex of maximum degree in $G'$.  If $d_{G'}(w)\ge \floor{n/2}$, then also $d_G(w)\ge d_{G'}(w)\ge \floor{n/2}$.
	Thus $d_G(v)+d_G(w)\ge \ceil{n/2}+3 + \floor{n/2} = n+3$. 
	So by Pigeonhole $v$ and $w$ have at least $3$ common neighbors, and $G$ contains $K_{2,3}$, a contradiction.
	Hence, $d_{G'}(w)\le \floor{n/2}-1$ and $\alpha_w(G')\ge 1 + (7n/8-(1+\floor{n/2}-1))/3\ge 1 + \floor{n/8}\ge \ceil{|G'|/7}$.

	Let $I'_w$ be an independent set in $G'$ that contains $w$ and has size $\floor{|G'|/7}$.  Let $G'':=G'-I'_w$.  
	Now we must show that $\Delta(G'') \le 2\ceil{|G''|/6}+3$, so we can finish by the Partitioning Lemma and the Tree Theorem, as we explain below.
	Suppose not, and let $x$ be a vertex of maximum degree in $G''$.  So 
	\begin{align*}
		d_G(v)+d_G(w)+d_G(x) &\ge \ceil{|G|/2}+3+2(2\ceil{|G''|/6}+4)\\
				     &\ge |G|/2 + 2|G''|/3 + 11\\
				     &\ge |G|/2 + 2(3|G|/4)/3 + 11\\
				     &= |G|+11.
	\end{align*}
	So by Pigeonhole, two of vertices $v,w,x$ share at least $3$ neighbors in $G$.  
	Thus, $G$ contains $K_{2,3}$, again contradicting that $G$ is outerplanar.  Hence, $\Delta(G'')\le 2\ceil{|G''|/6}+3$, as desired.

	By the Partitioning Lemma, $G$ has a forest equipartition $F_1\uplus F_2$ with $\Delta(F_i)\le \ceil{|G''|/6}+1$ for each $i\in [2]$.
	We will get an equitable $3$-coloring $\vph_i$ of each forest $F_i$; together $\vph_1$ and $\vph_2$ give an equitable $6$-coloring of $G$.
	Recall that $\alpha_v(F_i)$ denotes the size of a largest independent set in $F_i$ containing vertex $v$, for each $i\in[2]$ and each $v\in F_i$.
	To see that each $\vph_i$ exists, by the Tree Theorem it suffices to show that always $\alpha_v(F_i)\ge \floor{|F_i|/3}$.  Since $F_i$ is a forest,
	each of its subgraphs $H$ has an independent set of size at least $|H|/2$.
	Thus, we have $\alpha_v(F_i)\ge 1+\ceil{(|F_i| - (\Delta(F_i)+1))/2} \ge \ceil{(|F_i|-\ceil{|G''|/6})/2}$.
	To see that this final expression is at least $\floor{|F_i|/3}$, 
	let $k:=\floor{|G''|/6}$.  If $|G''| = 6k+5$, then $|F_i|/3 = k+1$ for some $i\in[2]$.  And our final expression is 
	$\ceil{((3k+3)-(k+1))/2} = k+1$.  Otherwise, $\floor{|F_i|/3} = k$ and $\ceil{(|F_i|-\ceil{|G''|/6})/2} \ge \ceil{(3k-(k+1))/2} = \ceil{(2k-1)/2}=k$, 
	as desired.
\end{proof}

\section{The Partitioning Lemma}
\label{partitioning-sec}
Recall from the introduction that a vertex $w$ is \emph{dangerous} if $d(w)\ge 2\lceil n/6\rceil + 4$.

\begin{lem}
	If $G$ is outerplanar, then $G$ has at most $2$ vertices $w$ such that $d(w)\ge 2\lceil n/6\rceil + 3$.
	In particular, $G$ has at most $2$ dangerous vertices.
	\label{few-big-lem}
\end{lem}
\begin{proof}
	Suppose instead that $G$ has at least $3$ such vertices: $w_1, w_2, w_3$.  Now $d(w_1)+d(w_2)+d(w_3) \ge n+9$.
	Since $G$ has no copy of $K_{2,3}$, at most one vertex is adjacent to all $3$ of $w_1,w_2,w_3$.  
	So the remaining at least $n+6$ incident edges go to at most $n-1$ vertices.
	Thus, at least $7$ vertices receive edges from at least $2$ of $w_1,w_2,w_3$.  By Pigeonhole, the vertices in one of these $3$ pairs 
	among $w_1,w_2,w_3$ have at least $\lceil 7/3\rceil = 3$ common neighbors.  Hence, $G$ contains a copy of $K_{2,3}$, contradicting that $G$ 
	is outerplanar.
\end{proof}

To prove the Main Theorem, we handle seperately the cases that the number of dangerous vertices in $G$ is $0$, $1$, or $2$.
In this section we prove the Partitioning Lemma, restated below, which combines with the Tree Theorem to handle the case that
the number of dangerous vertices in $G$ is $0$.  When applying the Partitioning Lemma, we are mainly intersested in the case that 
$\Delta(G)\le 2\ceil{n/6}+3$, in which case the degree bound simplifies to $d_{F_i}(v)\le \ceil{n/6}+1$.  But by stating the lemma 
more generally, we can apply it directly in our proof above of the Main Theorem when $s\ge 8$.

\begin{lem}[Partitioning Lemma]
	If $G$ is outerplanar, then $G$ has a forest equipartition $F_1\uplus F_2$ where 
	$d_{F_i}(v)\le \max\{\ceil{n/6}+1,\floor{d_G(v)/2}\}$ for all $i\in[2]$ and all $v\in F_i$.
	\label{partitioning-lem}
\end{lem}
\begin{cor}
	If $\Delta(G)\le 2\ceil{n/6}+3$ (so $G$ has $0$ dangerous vertices), then $G$ has an equitable $6$-coloring.
\end{cor}
\begin{proof}[Proof Sketch]
	The proof is nearly identical to the final paragraph in the proof of \Cref{main-thm-8}.
	We get an equipartition into $2$ forests, and must show that each forest $F_i$ has an equitable $3$-coloring.
	This follows directly from the fact that $\Delta(F_i)\le \ceil{n/6}+1$.
\end{proof}

\begin{proof}[Proof of the Partitioning Lemma]
	We begin by outlining the proof.  By \Cref{few-big-lem}, $G$ has at most $2$ vertices, say $w_1$ and $w_2$ such that 
	$d_G(w_i)\ge 2\ceil{n/6}+3$.  First we add edges to get a maximal outerplanar supergraph $G'$ of $G$ (with the same vertex set)
	such that $d_{G'}(v)\le 2\ceil{n/6}+3$ for all $v\in V(G)$ except $d_{G'}(w_i)\le d_G(w_i)+1$.
	Second, we get a forest equipartition $F_1\uplus F_2$ for $G'$ such that $\Delta(F_i)\le \ceil{n/6}+1$.
	Mainly, we show that each forest inherits at most half of the edges incident to each of its vertices; but we must be more careful with
	$w_1,w_2$, if they exist.  Finally, we show that each forest $F_i$ has an equitable $3$-coloring $\vph_i$.  Together, $\vph_1$
	and $\vph_2$ give an equitable $6$-coloring of $G$.  Now we provide the details.  

\begin{clm}
	If $G$ is outerplanar, then there exists a supergraph $G'$ of $G$ such that $G'$ is maximal outerplanar and $d_{G'}(v)\le 2\lceil n/6\rceil+3$ 
	for all vertices $v$ but at most two, say $w_1,w_2$.  Furthermore, $d_{G'}(w_i)\le \max\{2\ceil{n/6}+4,d_G(w_i)+1\}$ and if 
    	equality holds for both $i\in [2]$, then $w_1w_2\in E(G')$.
	\label{clm1}
\end{clm}
\begin{clmproof}[Proof Sketch]
	We just sketch the main idea; the full details are available in the appendix.  A vertex $v$ is \emph{bad} if $d_G(v)\ge 2\ceil{n/6}+3$.
	To form $G'$ from $G$, we add edges as long as possible, maintaining outerplanarity.
	The key is to avoid adding edges incident to bad vertices, of which (by \Cref{few-big-lem}) we have at most $2$.
	If $G$ has at most $1$ bad vertex, say $w$, then we never need to add an edge incident to $w$.  However, by adding edges we might create
	a second bad vertex.  So assume that $G$ has exactly $2$ bad vertices, $w_1$ and $w_2$.  

	As we form $G'$, only $2$ situations can force us to add an edge incident to a bad vertex $w_i$.  The first is when $w_1,w_2$ appear
	consecutively along a $4$-face that is not the outer face.  The second is when $w_1$ and $w_2$ have a common neighbor that is a cut-vertex.
	The first type can occur at most twice, and the second type can occur at most once; but in total they can occur at most twice.
\end{clmproof}

	Our plan now is to partition $V(G')$ as $F_1\uplus F_2$ by induction on $|G'|$.  Essentially, the induction is driven by the fact that every
	outerplanar graph is $2$-degenerate; actually, we will use something slightly stronger.  However, this induction argument is complicated by
	the presence of vertices $w_1,w_2$.  To ensure that we can save a bit extra for each $w_i$, we want them to appear last in the degeneracy order, 
	so first in the inductive construction of $F_1\uplus F_2$; this extra requirement motivates the ``furthermore'' statement in the next claim.
	(We prove the next claim for a general maximal outerplanar graph $H$, but we will specifically apply it to the graph $G'$ resulting from the 
	previous claim.)  As usual, a \emph{$k$-vertex} is a vertex of degree $k$.  Similarly, a \emph{$k$-neighbor}, of a specified vertex, is a neighbor
	of degree $k$.

\begin{clm}
	If $H$ is a maximal outerplanar graph with $|H|\ge 4$, then $H$ contains either (a) a $3$-vertex with a $2$-neighbor or (b) a $4$-vertex with two 
	$2$-neighbors.  Furthermore, if $|H|\ge 5$, then we can specify an arbitrary edge $e$ of $H$ and find either (a) or (b) that contains neither 
	endpoint of $e$.
	\label{clm2}
\end{clm}

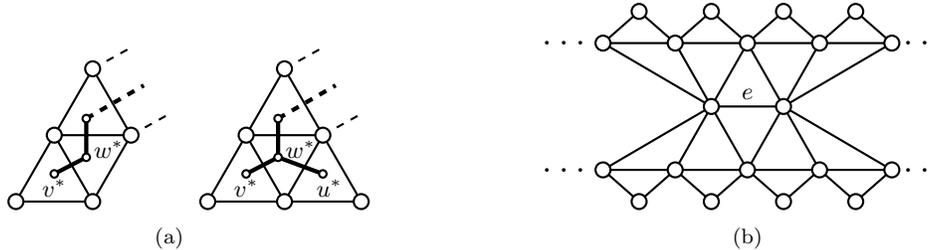
\begin{figure}[!b]
\centering
\begin{subfigure}{0.45\textwidth}
\centering
\begin{tikzpicture}[scale=.85, yscale=.866]
\begin{scope}[scale=1.2]
\draw[thick] (0.5,1) -- (0,0) -- (1,0) -- cycle (0.5,1) -- (1.5,1) -- (1,0) (0.5,1) -- (1,2) -- (1.5,1); 
\draw[thick] (1,2) edge[dashed] (1.5,2.32) (1.5,1) edge[dashed] (2,1.32);
\foreach \i/\j in {1.5/1, 0/0, 1/0, 0.5/1, 1/2}
\draw[thick] (\i,\j) node[uStyle] {};
\end{scope}
\begin{scope}[every node/.style={scale=0.5}, xshift=0.2cm, yshift=0.2cm]
\draw[ultra thick] (0.4,0.3) -- (0.9,0.6) -- (0.9,1.3) (0.9,1.3) edge[dashed] (1.8,1.90);
\foreach \i/\j in {0.4/0.3, 0.9/0.6, 0.9/1.3}
\draw[thick, fill=black] (\i,\j) node[uStyle] {};
\end{scope}
\begin{scope}
	\draw[thick] (0.6,0.25) node {\footnotesize{$v^*$}} (1.45,0.95) node {\footnotesize{$w^*$}};
\end{scope}

\begin{scope}[xshift=3cm, scale=1.2]
\draw[thick] (0.5,1) -- (0,0) -- (1,0) -- cycle (0.5,1) -- (1.5,1) -- (1,0) (0.5,1) -- (1,2) -- (1.5,1) 
	(1,2) edge[dashed] (1.5,2.32) (1.5,1) edge[dashed] (2,1.32)
	(1,0) -- (2,0) -- (1.5,1);
\foreach \i/\j in {1.5/1, 0/0, 1/0, 0.5/1, 1/2, 2/0}
\draw[thick] (\i,\j) node[uStyle] {};
\end{scope}
\begin{scope}[every node/.style={scale=0.5}, xshift=0.2cm, yshift=0.2cm, xshift=3cm]
\draw[ultra thick] (0.4,0.3) -- (0.9,0.6) -- (0.9,1.3) (0.9,1.3) edge[dashed] (1.8,1.9) (0.9,0.6) -- (1.6,0.3);
\foreach \i/\j in {0.4/0.3, 0.9/0.6, 0.9/1.3, 1.6/0.3}
\draw[thick, fill=black] (\i,\j) node[uStyle] {};
\end{scope}
\begin{scope}[xshift=3cm]
	\draw[thick] (0.6,0.25) node {\footnotesize{$v^*$}} (1.45,0.95) node {\footnotesize{$w^*$}} (1.9,0.25) node {\footnotesize{$u^*$}};
\end{scope}
\end{tikzpicture}
\caption{\textcolor{white}{the}}
\label{claim3-fig-a}
\end{subfigure}%
\begin{subfigure}{0.45\textwidth}
\centering
	\begin{tikzpicture}[scale=.80, yscale=1.255]
		\def\ep{.5}
\begin{scope}[scale=1.2, yscale=.7, yshift=.4cm]
\draw[thick] 
	(-0.5,1) -- (0.5,1) 
	(-0.5,1) -- (0,0) -- (0.5,1) 
	(-0.5,1) -- (0,2) -- (0.5,1) 
	(-1,2) -- (0,2) -- (1,2) 
	(-1,0) -- (0,0) -- (1,0) 
	(-0.5,1) -- (-1,2) -- (-2,2) --cycle 
	(0.5,1) -- (1,2) -- (2,2) -- cycle 
	(-0.5,1) -- (-1,0) -- (-2,0) -- cycle 
	(0.5,1) -- (1,0) -- (2,0) -- cycle;
\draw[thick] 
	(-2,2) -- (-1.5,2+\ep) node[uStyle] {} -- (-1,2)
	(-1,2) -- (-.5,2+\ep) node[uStyle] {} -- (0,2)
	(0,2) -- (.5,2+\ep) node[uStyle] {} -- (1,2)
	(1,2) -- (1.5,2+\ep) node[uStyle] {} -- (2,2);
\draw[thick] 
	(-2,0) -- (-1.5,-\ep) node[uStyle] {} -- (-1,0)
	(-1,0) -- (-.5,-\ep) node[uStyle] {} -- (0,0)
	(0,0) -- (.5,-\ep) node[uStyle] {} -- (1,0)
	(1,0) -- (1.5,-\ep) node[uStyle] {} -- (2,0);
\foreach \i/\j in {-0.5/1, 0/0, 0.5/1, 0/2, -1/2, 1/2, -1/0, 1/0, -2/2, 2/2, -2/0, 2/0}
\draw[thick] (\i,\j) node[uStyle] {};
\draw[thick] (-2.5,2) node {\Large{$\dots$}} (2.5,2) node {\Large{$\dots$}} (-2.5,0) node {\Large{$\dots$}} (2.5,0) node {\Large{$\dots$}} (0,1.20) node {\footnotesize{$e$}}; 
\end{scope}
\end{tikzpicture}
\caption{}
\label{claim3-fig-b}
\end{subfigure}
\caption{(a): The end $v^*w^*$ of a longest path in the weak dual $G^*$ (bold) of $G$ (non-bold), where $d(w^*)=2$ (left) and $d(w^*)=3$ (right). Dashed edges indicate that the graph extends further in some unknown fashion. (b):~The structure of $G$ in the final case of \Cref{clm2}.}
\end{figure}

\begin{clmproof}
	We consider the end of a longest path $P$ in the weak planar dual $H^*$ of $H$.  (Recall that $V(H^*)$ consists of the $3$-faces of $H$ and that 
	$2$ vertices are adjacent in $H^*$ if their corresponding $3$-faces share an edge.  So $H^*$ is a tree with $\Delta(H^*)\le 3$.)  Let $v^*,w^*$ 
	be the final $2$ vertices at some end of $P$, with $v^*$ last.  Clearly, $d_{H^*}(v^*)=1$.  If $d_{H^*}(w^*)=2$, then face $v^*$ of $H^*$ 
	contains an adjacent $3$-vertex and $2$-vertex; see Figure~\ref{claim3-fig-a} (left).
	Suppose instead that $d_{H^*}(w^*)=3$.  Now $2$ of the $3$ neighbors of $w^*$, say $v^*$ and $u^*$, must be leaves in $H^*$, since $w^*$ is the
	penultimate vertex on a longest path in $H^*$. Thus the vertex common to $v^*$ and $u^*$ is a $4$-vertex in $H$ and it has (distinct) 
	$2$-neighbors on the faces corresponding to both $u^*$ and $v^*$; see Figure~\ref{claim3-fig-a} (right).

	Finally, suppose that $|H|\ge 5$, and fix an arbitrary edge $e$.  Our main idea is to find an endpoint $v^*$ of a longest path $P$ in $H^*$ 
	such that the $3$-face corresponding to $v^*$ contains no endpoint of $e$.  Suppose instead that every instance of (a) and (b) 
	contains an endpoint of $e$.  We claim that $H$ is an induced subgraph of a graph with the form 
	shown in Figure~\ref{claim3-fig-b}.  Denote the endpoints of $e$ by $\{w_1,w_2\}$. 
	Suppose instead that a component $C$ of $H-\{w_1,w_2\}$ has at least $2$ vertices.  Since $H$ is maximal outerplanar, it is $2$-connected, 
	so $C$ has $2$ neighbors in $N[\{w_1,w_2\}]$; call them $x_1$, $x_2$.  By symmetry, we assume $x_1,x_2\in N(w_1)$. 
	Now we proceed by induction on $H[V(C)\cup\{w_1,x_1,x_2\}]$ with specified edge $x_1x_2$.  This proves the claim.  Thus, the leftmost or 
	rightmost vertex in $H$ lies in the desired copy of (a) or (b).
\end{clmproof}

\Cref{clm2} implies the following helpful lemma: If $H$ is outerplanar, then $H$ contains either (a) a $1^-$-vertex or (b) a $3^-$-vertex with a 
$2$-neighbor or (c) a $4$-vertex with two $2$-neighbors.  This result was first proved by Hackmann and Kemnitz~\cite{HK} and the proof was later 
simplified by Fabrici~\cite{fabrici}.  Our proof follows the same approach; however, later our added ability to avoid the endpoints of edge $e$ will 
prove crucial.

    \begin{figure}[!h]
        \centering
	    \begin{tikzpicture}[yscale=1.732,scale=.9]
        \begin{scope}
		\draw[thick] (0,0) node[uStyle, fill=\mgray] (a) {} -- (1,0) node[uStyle] {} -- (0.5,-0.5) node[uStyle, fill=\mgray] (b) {};
		\drawthick (a) edge (b);
        \end{scope}

        \begin{scope}[xshift=2.5cm]
		\draw[thick] (0.5,-0.5) edge (0,0) (0,0) node[uStyle] (a) {} -- (1,0) node[uStyle, fill=\mgray] {} -- (0.5,-0.5) node[uStyle] (b) {} -- (-0.5,-0.5) node[uStyle, fill=\mgray] {} -- cycle; 
		\drawthick (a) -- (b);
        \end{scope}

        \begin{scope}[xshift=5cm]
		\draw[thick] (1,0) edge (0.5,-0.5) (0.5,-0.5) -- (0,0) (0,0) node[uStyle, fill=\mgray] (a) {} -- (1,0) node[uStyle] (b) {} -- (1.5,-0.5) node[uStyle, fill=\mgray] {} -- (0.5,-0.5) node[uStyle] {} -- (-0.5,-0.5) node[uStyle, fill=\mgray] (d) {}; 
	\drawthick (a) edge (d); 
        \end{scope}

        \begin{scope}[xshift=8cm]
		\draw[thick] (0,0) edge (0.5,0.5) (0.5,0.5) node[uStyle] (a2) {} -- (1,0) (-0.5,-0.5) node[uStyle, fill=\mgray] (a) {} -- (0,0) -- (0.5,-0.5) (0,0) node[uStyle] (b2) {} -- (1,0) node[uStyle, fill=\mgray] (b) {} -- (1.5,-0.5) node[uStyle] {} -- (0.5,-0.5) node[uStyle, fill=\mgray] (c) {} 
		(0.5,0.1) node {\footnotesize{$e_1$}} (1.5,-0.25) node {\footnotesize{$e_2$}};
		\drawthick (a) -- (c) -- (b) (a2) -- (b2); 
        \end{scope}

        \begin{scope}[xshift=11.5cm]
		\draw[thick] (-.5,-.5) -- (-1,0) node[uStyle] (a1) {} -- (0,0) edge (0.5,0.5) (0.5,0.5) node[uStyle] (a2) {} -- (1,0) 
		(-0.5,-0.5) node[uStyle, fill=\mgray] (a) {} -- (0,0) -- (0.5,-0.5) (0,0) node[uStyle, fill=\mgray] (b2) {} 
		-- (1,0) node[uStyle, fill=\mgray] (b) {} (0.5,-0.5) node[uStyle] (c) {} (a) -- (c) -- (b); 
		\drawthick (a) -- (b2) -- (b);
        \end{scope}

        \end{tikzpicture}
        \captionsetup{width=.675\textwidth}
        \caption{The 5 base cases for Claim~\ref{clm3}. Vertices of $F_1$ are shown in white, and those of $F_2$ are shown in gray.  Edges induced by each $F_i$ are drawn in bold.}
        \label{base-case-fig}
    \end{figure}
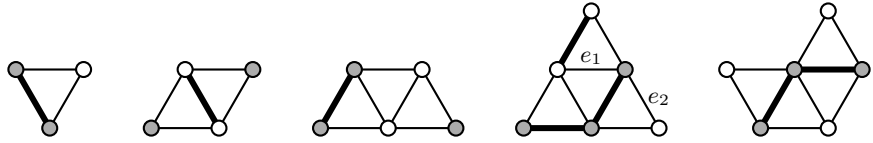
\begin{clm}
	Let $H$ be a maximal outerplanar graph (so with $|H|\ge 3$), with a specified edge $e$.
	There exists a forest equipartition $F_1\uplus F_2$ for $H$ such that 
	$d_{F_i}(v)\le d_H(v)/2$ for each $i\in[2]$ and all $v\in F_i$.  Furthermore, if $|H|\ge 4$, then for each endpoint $w$ of $e$ and 
	$F_i\ni w$, we have $d_{F_i}(w)\le (d_H(w)-1)/2$.
	\label{clm3}
\end{clm}
\begin{clmproof}
	We use induction on $|H|$.  
	Our base case is when $|H|\in\{3,4,5,6\}$.  
	The 5 possibilities for $H$ are shown in Figure~\ref{base-case-fig}, along with the corresponding forest equipartitions.
	We remark on a few of the cases.
	If $|H|=3$, then it suffices to have $d_{F_i}(v)\le 1 = d_H(v)/2$; that is, we do not treat the endpoints of $e$ differently from the other vertex.
	If $|H|=5$, then by symmetry we assume that $e$ is not incident to the leftmost vertex.
	If $|H|=6$ and $\Delta(H)=4$, then by symmetry we assume that $e$ is either $e_1$ or $e_2$, as labeled in the figure.
	Now we turn to the induction step; for this we consider (a) and (b) in \Cref{clm2}.

    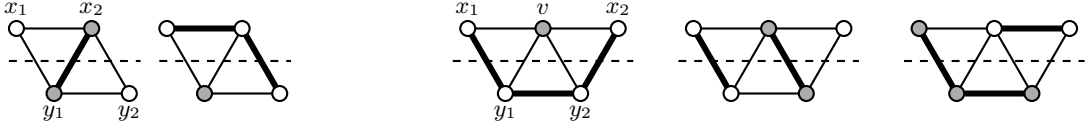
\begin{figure}[!t]
        \centering
        \begin{tikzpicture}[yscale=1.73]
        \begin{scope}
		\draw[thick] (0,0) node[uStyle] {} -- (1,0) node[uStyle, fill=\mgray] (a) {} -- (1.5,-0.5) node[uStyle] {} -- (0.5,-0.5) node[uStyle, fill=\mgray] (b) {} -- cycle (0,0.15) node {\footnotesize{$x_1$}} (1,0.15) node {\footnotesize{$x_2$}} (1.5,-0.65) node {\footnotesize{$y_2$}} (0.5,-0.65) node {\footnotesize{$y_1$}}; 
		\drawthick (a) -- (b); 
		\draw[dashed, thick] (-.1,-.25) -- (1.65,-.25);
         \end{scope}

         \begin{scope}[xshift=2cm]
		 \draw[thick] (1,0) -- (0.5,-0.5) (0,0) node[uStyle] (a) {} -- (1,0) node[uStyle] (b) {} -- (1.5,-0.5) node[uStyle] (c) {} -- (0.5,-0.5) node[uStyle, fill=\mgray] {} -- cycle; 
		 \drawthick (a) -- (b) -- (c);
		 \draw[dashed, thick] (-.1,-.25) -- (1.65,-.25);
         \end{scope}

         \begin{scope}[xshift=6cm]
		 \draw[thick] (1,0) -- (2,0) node[uStyle] (a) {} -- (1.5,-0.5) (1,0) -- (0.5,-0.5) (0,0) node[uStyle] (b) {} -- (1,0) node[uStyle, fill=\mgray] {} -- (1.5,-0.5) node[uStyle] (c) {} -- (0.5,-0.5) node[uStyle] (d) {} -- cycle (0,0.15) node {\footnotesize{$x_1$}} (1,0.15) node {\footnotesize{$v$}} (1.5,-0.65) node {\footnotesize{$y_2$}} (0.5,-0.65) node {\footnotesize{$y_1$}} (2,0.15) node {\footnotesize{$x_2$}}; 
		 \drawthick (a) -- (c)  -- (d) -- (b);
		 \draw[dashed, thick] (-.2,-.25) -- (2.2,-.25);
         \end{scope}

         \begin{scope}[xshift=9cm]
		 \draw[thick] (1,0) -- (2,0) node[uStyle] (a) {} -- (1.5,-0.5) (1,0) -- (0.5,-0.5) (0,0) node[uStyle] (b) {} -- (1,0) node[uStyle, fill=\mgray] (c) {} -- (1.5,-0.5) node[uStyle, fill=\mgray] (d) {} -- (0.5,-0.5) node[uStyle] (e) {} -- cycle; 
		 \drawthick (b) -- (e) (c) -- (d); 
		 \draw[dashed, thick] (-.2,-.25) -- (2.2,-.25);
         \end{scope}

         \begin{scope}[xshift=12cm]
		 \draw[thick] (1,0) -- (2,0) node[uStyle] (a1) {} -- (1.5,-0.5) (1,0) -- (0.5,-0.5) (0,0) node[uStyle, fill=\mgray] (a) {} -- (1,0) node[uStyle] (b1) {} -- (1.5,-0.5) node[uStyle, fill=\mgray] (b) {} -- (0.5,-0.5) node[uStyle, fill=\mgray] (c) {} -- cycle; 
		 \drawthick (b) -- (c) -- (a) (a1) -- (b1); 
		 \draw[dashed, thick] (-.2,-.25) -- (2.2,-.25);
         \end{scope}
        \end{tikzpicture}
	    \caption{The 5 cases in the two induction steps for \Cref{clm3}. Again vertices of $F_1$ (resp.~$F_2$) are shown in white  (resp.~gray) and edges induced by each $F_i$ are drawn in bold.  Vertices below the dashed line are already assigned to one forest or the other (along with possible additional vertices that are unshown), and vertices above the dashed line are now assigned to one forest or the other, as part of the induction step.}
        \label{induction-step-fig}
    \end{figure}

	First we consider (a) in \Cref{clm2}; see the left of \Cref{induction-step-fig}.  Denote the $2$-vertex by $x_1$ and the $3$-vertex by $x_2$.  
	Since $H$ is maximal outerplanar, vertices $x_1,x_2$ have a common neighbor and we denote it by $y_1$; we denote the final neighbor of $x_2$ by $y_2$.  
	Let $H':=H-\{x_1,x_2\}$.  By the induction hypothesis, $H'$ has a forest equipartition $F_1'\uplus F_2'$ such that 
	$d_{H'[F_i]}(v)\le d_H'(v)/2$ for each $i\in [2]$ and each $v\in F'_i$.  By symmetry, we assume that $y_2\in F_1'$.
	Now let $F_1:=F_1'\cup\{x_1\}$ and let $F_2:=F_2'\cup\{x_2\}$.  Clearly, $F_1\uplus F_2$ is a forest equipartition of $V(H)$.  
	It is also straightforward to check that the degree bounds hold, since $d_{F_i}(x_i)\le 1$ for each $i\in[2]$, and
	each $y_i$ inherits at most one additional edge in its forest.

	Now we consider (b) in \Cref{clm2}; see the right of \Cref{induction-step-fig}.  We denote the $4$-vertex by $v$ and its two $2$-neighbors 
	by $x_1$ and $x_2$.  Since $H$ is maximal outerplanar, we assume that $v,x_1$ have a common neighbor $y_1$; and that $v,x_2$ have
	a common neighbor $y_2$; furthermore, $y_1y_2\in E(H)$.  Let $H':=H-\{v,x_1,x_2\}$ and let $F_1'\uplus F_2'$ be the forest equipartition of $V(H')$
	guaranteed by the induction hypothesis.  By symmetry, we assume that $|F_1'|\le |F_2'|$, so we will need to add two of $v,x_1,x_2$ to $F_1'$
	and add one to $F_2'$.  If $y_1\in F_1'$, then let $F_1:=F_1'\cup \{x_1,x_2\}$ and let $F_2:=F_2'\cup\{v\}$.  So we assume that $y_1\in F_2'$;
	by symmetry between $y_1$ and $y_2$, we also assume that $y_2\in F_2'$.  Now we let $F_1:=F_1'\cup \{v,x_2\}$ and let $F_2:=F_2'\cup\{x_1\}$.
	Clearly, $F_i$ is a forest for each $i\in[2]$.  Furthermore, each of $v,x_1,x_2$ has at most one incident edge in its forest, and each $y_i$
	has at most one additional incident edge in its forest.  Thus, the desired degree bounds hold.
\end{clmproof}

By \Cref{clm1} we add edges to $G$ to get a maximal outerplanar graph $G'$ with 
$d_{G'}(v)\le 2\ceil{n/6}+3$ for all vertices $v$, but at most two: $w_1,w_2$; and $d_{G'}(w_i)\le \max\{2\ceil{n/6}+4,d_G(w)+1\}$ 
for each $w_i$.
Furthermore, if equality holds for each $i\in[2]$, then $w_1w_2\in E(G')$.
By \Cref{clm3}, with $e=w_1w_2$, we get a forest equipartition $F_1\uplus F_2$ for $G'$ with 
$d_{F_i}(v)\le \max\{\ceil{n/6}+1,\floor{d_G(v)/2}\}$ for all $i\in[2]$ and all $v\in F_i$.
(If $w_1$ exists, but $w_2$ does not, then we choose $e$ to be an arbitrary edge incident with $w_1$.)
\end{proof}

In the introduction, we promised to sketch the proof of Pemmaraju's equipartitioning result.  It follows from simplified versions of Claims~2 and~3.
In particular, in Claim~2 we do not need edge $e$ (so it suffices to use the earlier version due to Hackmann and Kemnitz), 
and in Claim~3 we do not guarantee any bounds on the degrees.  We just inherit the bound $d_{F_i}(v)\le \Delta(G)$.

\section{Proving the Main Theorem}
\label{s=6-sec}

By \Cref{few-big-lem}, the number of dangerous vertices in $G$ is either $0$, $1$, or $2$.
The first case is (now) easy; we get an equipartition into $2$ forests, and equitably $3$-color each.
In this section we handle the remaining $2$ cases, which finishes the proof of the Main Theorem.

\begin{defn}
	\label{n_i-remark}
In what follows we often explicitly construct an equitable $6$-coloring.  Thus, we let $n_j:=\floor{(n+j-1)/6}$, for each $j\in[6]$.
We partition $V(G)$ as $I_1\uplus \cdots \uplus I_6$ with $|I_j|=n_j$, and use color $j$ on $I_j$.  Note that $\sum_{j=1}^6 n_j = n$.
	To see this, let $k:=\floor{n/6}$, and let $\ell:=n-6k$.  Now $\sum_{j=1}^6n_j=\sum_{j=1}^6\floor{(6k+\ell+j-1)/6} 
	= 6k+\sum_{j=0}^5\floor{(j+\ell)/6} = 6k+\sum_{j=0}^{5-\ell}0+\sum_{j=6-\ell}^51 = 6k+\ell=n$, as desired.
\end{defn}

In the remainder of the proof, we assume that $G$ is maximal outerplanar; this is possible by Claim~1~in the proof of the Partitioning Lemma.
Formally, we begin the proof of the Main Theorem with an arbitrary outerplanar graph $G$, and extend $G$ to a maximal outerplanar graph $G'$
by Claim~1.  If $\Delta(G')\le 2\ceil{n/6}+4$ and all vertices achieving this bound are endpoints of a common edge, then we proceed via Claims $2$ and~$3$
in the proof of the Partitioning Lemma, as in Section~2, finishing with an equitable $3$-coloring of each $F_i$.  Otherwise, after applying Claim~1, we
have $2$ or $1$ dangerous vertices and finish via the next $2$ subsections.

We assume implicitly that we are given an outerplane embedding of $G$.
We often use the following easy observation. 

\begin{lem}
	Each outerplanar graph $H$ has a $3$-coloring $\vph$ with $|\vph^{-1}(j)|\le |H|/2$ for all $j\in[3]$.
	That is, each color class has size at most $|H|/2$.
	\label{small-classes-obs}
\end{lem}
\begin{proof}
In every $3$-coloring of a maximal outerplanar graph $H'$, each color class has size at most $|H'|/2$, since the boundary of the outer face is a 
Hamiltonian cycle, which has independence number $\floor{|H'|/2}$.  
If $H$ is not maximal outerplanar, then we add edges (but no vertices) to reach a maximal outerplanar supergraph $H'$, then $3$-color $H'$ and
deduce the result for $H$.
\end{proof}

\subsection{Two Dangerous Vertices}
\label{two-dangerous-sec}
We first handle the case that $G$ has $2$ dangerous vertices.  Actually, we handle the slightly more general case that there exist vertices 
$w_1$ and $w_2$ such that $d(w_i)\ge 2\ceil{n/6}+3$.  This generality will help us in the next, and final, subsection of the proof.

\begin{figure}[!b]
    \centering
    \begin{tikzpicture}
    \begin{scope}
    \draw[thick] (6,0) -- (6,-1) -- (5,0) (6,-2) -- (6,-1) -- (7,-2) (7,0) -- (7,-1) -- (7,-2) (7,-1) -- (6,0) (5.5,2) edge[bend right=45] (5.5,-4) 
	    (1,0) -- (5.5,2) -- (10,0) (4,0) -- (4.2,0) (5,0) -- (4.7,0); 
    \foreach \i in {1,...,3}
    \draw[thick] (\i,0) -- (\i+1,0);
    \foreach \i in {5,...,9}
    \draw[thick] (\i,0) -- (\i+1,0);
    \foreach \i in {2,4,6,8,10}
    \draw[thick] (\i,0) node[uStyle, fill = black] {};
    \foreach \i in {1,3,7,9}
    \draw[thick] (\i,0) node[uStyle, fill = myothergrayer] {}; 
	    \draw[thick] (5,0) node[xStyle] {} (5,0) -- (5,-2) (6,-1) node[yStyle] {} (7,-1) node[yStyle] {} (5.7,-1) node {\footnotesize{$y_1$}} (7.3,-1) node {\footnotesize{\,$y_2$}} (5.5,2) node[xStyle] {} (5.5,2.3) node {\footnotesize{$w_1$}};
    \end{scope}

    \begin{scope}[yshift=-2cm]
    \draw[thick] (4,0) -- (4.2,0) (5,0) -- (4.7,0) (1,0) -- (5.5,-2) -- (10,0);
    \foreach \i in {1,...,3}
    \draw[thick] (\i,0) -- (\i+1,0);
    \foreach \i in {5,...,9}
    \draw[thick] (\i,0) -- (\i+1,0);
    \foreach \i in {2,4,6,8,10}
    \draw[thick] (\i,0) node[uStyle] {};
    \foreach \i in {1,3,7,9}
	    \draw[thick] (\i,0) node[uStyle, fill=mygrayer] {}; \draw[thick] (5,0) node[xStyle] {} (5.5,-2) node[xStyle] {} (5.5,-2.3) node {\footnotesize{$w_2$}};
    \end{scope}
    \end{tikzpicture}
    \captionsetup{width=.675\textwidth}
    \caption{An example of the graph in \Cref{2-dangerous-lem}. Square vertices are not part of $S'_1\cup S'_2$. The black and dark gray vertices partition $S'_1$ into $2$ independent sets (namely, $I_5\setminus\{w_2\}$ and $S''_1$. The white and light gray vertices partition $S'_2$ into $2$ independent sets (namely, $I_6\setminus\{w_1\}$ and $S''_2$. Diamond vertices do not necessarily belong to $N[w_1]\cup N[w_2]$. Edge $w_1w_2$ might or might not be present.}
    \label{two-dangerous-vrts-fig}
\end{figure}
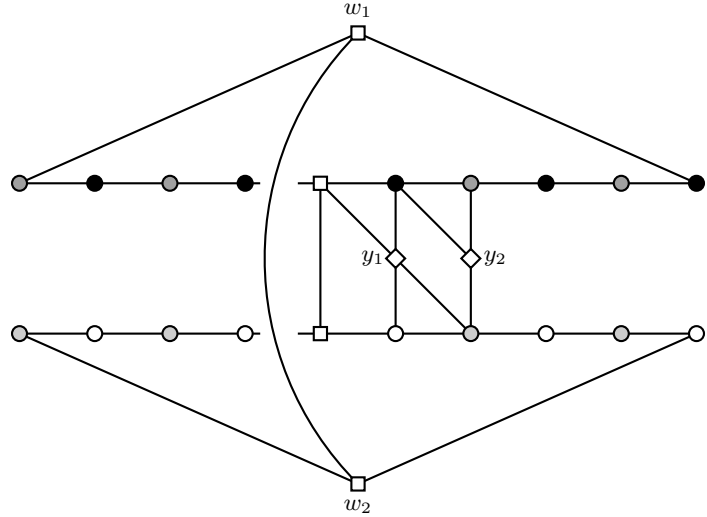

\begin{lem}
	If $G$ is outerplanar with vertices $w_1,w_2$ such that $d(w_i)\ge 2\ceil{n/6}+3$ for each $i\in[2]$, then $G$ has an equitable $6$-coloring.
	\label{2-dangerous-lem}
\end{lem}
\begin{proof}
	We begin with a proof sketch.
	Recall the definitions of $n_j$ and $I_j$ from \Cref{n_i-remark}, and that $I_j$ uses color $j$.  We pick $I_5$ within $\{w_2\}\cup N(w_1)$ and 
	pick $I_6$ within
	$\{w_1\}\cup N(w_2)$.  We $3$-color with colors $1,2,3$ a subgraph $G'$ of order $n_2+n_3$ that includes all of $G\setminus(N[w_1]\cup N[w_2])$.
	By \Cref{small-classes-obs}, each color class of $G'$ has size at most $\floor{|G'|/2} = \floor{(n_2+n_3)/2}=n_2$.  
	We permute colors (if needed) so that the biggest, second biggest, and smallest of these color classes of $G'$ use, respectively, 
	colors 3, 2, and 1.  Finally, we color the remaining uncolored vertices, which are all in $N(w_1)\cup N(w_2)$, with colors $1,2,3,4$.  

	For this final step, we want all these uncolored vertices to form an independent set; we call this set $S_1''\cup S_2''$.  
	Because color $4$ is only used on $S_1''\cup S_2''$, color $4$ can go on any subset of size $n_4$, so we assign color $4$ last.  
	The total number of vertices to be colored with $1,2,3$ is $n_1+n_2+n_3$;
	and $n_2+n_3$ of these are in $G'$.  So we need to assign colors $1,2,3$ to a total of $n_1$ vertices in $S_1''\cup S_2''$.  We can
	do this greedily, but this claim requires justification, which we now provide.
	\smallskip

	Let $S_i:=N(w_i)\setminus N[N[w_{3-i}]\setminus \{w_i\}]$, for each $i\in [2]$.  Note, since $G$ has no $K_{2,3}$-minor, that 
	$|S_i|\ge d(w_i)-3 \ge 2\ceil{n/6} = 2n_6$.  We pick a set of $n_1+n_5-1$ consecutive vertices in $S_1$, and call it $S_1'$.  
	Observe that the vertices of $S_1'$ might not be consecutive, among neighbors of $w_1$, due to vertices in $N[N[w_2]\setminus\{w_1\}]$ 
	being removed; but they are consecutive among $S_1$.
	Since $N(w_1)$ induces a linear forest, $S_1'$ can be partitioned into $2$ independent sets: one of size $n_1$ 
	(starting with the first vertex in $S_1'$), and the other of size $n_5-1$.  The latter of these will be colored $5$, and combine 
	with $w_2$ to form $I_5$; the former will be called $S_1''$ and will be colored with $1,2,3,4$.
	Similarly, we pick a set of $n_4+n_6-1$ consecutive vertices in $S_2$, and call it $S_2'$.
	Since $N(w_2)$ induces a linear forest, $S_2'$ can be partitioned into $2$ independent sets: one of size $n_4$ 
	(starting with the first vertex in $S_2'$), and the other of size $n_6-1$.  
	The latter of these will be colored $6$, and combine with $w_1$ to form $I_6$; the former will be called $S_2''$ and will be colored 
	with $1,2,3,4$.  However, we must be careful about how we pick $S_1'$ and $S_2'$, as we now describe. 

	We claim that we can choose $S_1'$ and $S_2'$ to ensure that at most one vertex, $y$, has more than a single neighbor in $S_1''\cup S_2''$.
	Furthermore, if $y$ exists, then it has exactly one neighbor in each of $S_1''$ and $S_2''$.  To see this, we first note that each vertex other
	than $w_i$ has at most one neighbor in $S_i''$, regardless of our choice of $S_i'$.  Suppose, to the contrary, that some vertex $z$ has at least 
	$2$ neighbors in $S_i''$; call them
	$x',x''$.  Now the $4$-cycle $w_ix'zx''$ separates $2$ neighbors of $w_i$, so at least one of them must not lie on the outer face, contradicting
	our assumption that we have an outerplane embedding of $G$. (This argument
	uses the fact that $S_i'$ consists of \emph{consecutive} vertices in $S_i$; since $|S_i| \ge 2n_6 > |S_i'|$, the set $S_i'$ cannot contain both 
	the first and last vertex of $S_i$, so neither can the set $S_i''$.)  

	Now suppose there exist some vertices $y_i$ that each have neighbors in
	both $S_1$ and $S_2$.  Since $G$ has no $K_{2,3}$-minor, there exist at most $2$ such $y_i$.  First suppose there exists only one; call it $y_1$.
	Again, since $G$ has no $K_{2,3}$-minor, $y_1$ has at most $2$ neighbors in each of $S_1$ and $S_2$.  If $y_1$ has $2$ neighbors in some choice 
	of $S_i'$, then they must be successive neighbors of $w_i$, so at most one of them will be in $S_i''$, and we are done.  So assume instead that 
	there exist $2$ such vertices, $y_1$ and $y_2$; see \Cref{two-dangerous-vrts-fig}.  The total number of neighbors that $y_1$ and $y_2$ have
	in $S_i$ is at most $2$, for each $i\in[2]$; otherwise $G$ has a $K_{2,3}$-minor. And at least one $y_j$ has a single neighbor in $S_i$, for the
	same reason.  Since we have at least $2$ choices for each $S_i''$, we can choose one such that $y_j$ has no neighbors in $S_i''$.  This proves the claim.

	Now we finish the equitable $6$-coloring.  Given the choices of $S_1'$ and $S_2'$ above, we use color $5$ on $\{w_2\}\cup (S_1'\setminus S_1'')$ and
	we use color $6$ on $\{w_1\}\cup (S_2'\setminus S_2'')$. We let $G':=G-(S_1'\cup S_2'\cup\{w_1,w_2\})$, and we $3$-color $G'$ with colors $1,2,3$.
	We permute colors (if needed) so that the biggest, second biggest, and smallest of these color classes of $G'$ use, respectively, colors 3, 2, and 1.
	Now we must color $n_1$ vertices of $S_1''\cup S_2''$ with colors $1,2,3$ so that color $j$ is used a total of $n_j$ times, for each $j\in[3]$.
	We do this greedily, simply picking a color $j$ that still needs to be used more times, picking an uncolored vertex $z$ in $S_1''\cup S_2''$ with 
	no neighbor colored $j$, coloring $z$ with $j$, and repeating.  

	So why does this greedy approach succeed?  The key is that every vertex in $G'$
	has at most one neighbor in $S_1''\cup S_2''$, with a possible single exception $y_1$ that could have two such neighbors.  So when we try to find
	a new vertex in $S_1''\cup S_2''$ on which to use color $j$, the number of vertices that are already colored is at most $n_1-1$, and the number of
	vertices that are ``blocked'' because they have a neighbor in $G'$ using color $j$ is at most $(n_j-1)+1=n_j$.  So the total number of vertices
	unavailable is at most $n_1-1+n_j\le n_1-1+n_3<|S_1''\cup S_2''| = n_1+n_4$; thus, some vertex in $S_1''\cup S_2''$ is available to receive color 
	$j$ as desired.  Lastly, we use color $4$ on the final $n_4$ remaining uncolored vertices of $S_1''\cup S_2''$.
\end{proof}

\subsection{One Dangerous Vertex}
\label{one-dangerous-sec}
Finally, suppose that $G$ contains a single dangerous vertex, $w$.

\begin{lem}
	If there exists $w\in V(G)$ with $d(w)\ge n/2$, 
	then $G$ has an equitable $6$-coloring.
	\label{alg1-lem1}
\end{lem}
\begin{proof}
	We now construct an equitable $6$-coloring with colors $1,\ldots,6$.  We use $I_j$ and $n_j$ as in \Cref{n_i-remark}, coloring $I_j$ with $j$.
	Let $I_1$ be an arbitrary independent set containing $w$ of size $\floor{n/6}$.  
	Let $G':=G-(N(w)\cup I_1)$.  Let $k:=\floor{n/6}$ and let $\ell:=n-6k$.
	Note that $|G'| = n-d(w)-n_1\le 6k+\ell - (3k+\ceil{\ell/2})-k = 2k + \floor{\ell/2}$.  By \Cref{small-classes-obs}, $G'$ has a $3$-coloring 
	$\vph'$ with each color class of size at most $\floor{|G'|/2} = \floor{(2k+\floor{\ell/2})/2} = k+\floor{\ell/4}$, and at most $2$ color classes
	achieve this size.  We color the biggest, middle, and smallest color classes of $\vph'$ with, respectively, colors $4,3,2$.
	Form coloring $\vph$ from $\vph'$ by also using color $1$ on $I_1$; see the left of \Cref{alg1-figs}. 
	Color $j$ is used on at most $n_j$ vertices, for each $j\in [6]$; thus far, colors $5$ and $6$ are unused.
	Now we extend our $4$-coloring of $G\setminus N(w)$ to an equitable $6$-coloring of $G$, via Algorithm~1 below.
	We denote the neighbors of $w$ by $x_1,\ldots,x_{d(w)}$ in order around $w$, with ${x_1x_{d(w)}}\notin E(G)$.

	Now we explain why this algorithm succeeds.  Throughout the algorithm, $n'_j$ denotes the number of additional vertices that must be colored with 
	$j$ to reach the desired total $n_j$.  Also, we let $n_j''$ denote the number of vertices in $G'$ colored $j$ with one or more neighbors $x_h$ 
	where $h\ge i$; we just use $n_j''$ to prove correctness.  First note that each time we attempt to color a vertex with $6$ we will succeed, 
	until we get 
	$n_6'=0$; we will show that at this point the algorithm halts because all vertices are colored (equivalently, because now $i=d(v)+1$).  This is 
	because no vertices of $G'$ are colored with $6$, and each time we try to color $x_i$ with $6$ we have just colored $x_{i-1}$ with $5$, so no 
	neighbor of $w_i$ uses color $6$.

	We maintain the invariant that each time a vertex is colored $6$ we have that $n_j'+n_j''\le n_6'$ for all $j\in[5]$.  This holds initially 
	because $n_6'=n_6\ge n_j\ge n_j'+n_j''$ for all $j\in[5]$.  And each time that we color a vertex $6$ it is after we try to color a vertex $j$.  
	If we succeed in coloring with $j$, then $n'_j$ decreases, so the invariant is maintained.  But if we fail in coloring with $j$, then it is 
	because the vertex we try to color with $j$ has a neighbor colored $j$ in $G'$.  (Note that because $G$ is outerplanar, each vertex in $G'$ 
	colored $j$ is adjacent to at most one vertex $x_i$ that we tried to color $j$.)  Thus, the invariant is again maintained.

	We remark that line 6 never has any effect.  This is because throughout the algorithm $n_5'=n_6'$, except for right after a vertex is colored $5$.  
	However, whenever that happens, we immediately move to the next color, which is $6$.  The reason that we include line 6 is that we use the 
	same algorithm for the proof of \Cref{alg1-lem2}, where it does have an effect.
\end{proof}

\begin{minipage}{13.6cm}
\begin{algorithm}[H]
    \caption{Extending a partial coloring to an equitable $6$-coloring of $G$}
    \label{euclid}
    \begin{algorithmic}[1] 
	    \STATE $i:=1$;~~$j:=1$ if $n_6=n_5$, and else $j:=6$;~~$n_j':=n_j-|\vph^{-1}(j)|$ for all $j\in[6]$ 
	    \WHILE{$i\le d(w)$} 
		\STATE try to color $x_i$ with color $j$ \hspace{.56in}\{This could fail either because $w_i$\\ 
		\hspace{2.48in}has a neighbor colored $j$ or because\\ 
		\hspace{2.48in}already $n_j$ vertices are colored $j$.\}
		\STATE \textbf{if} line 3 succeeds, \textbf{then} $n'_j\!:=\!n'_j\!-\!1$ 
		\STATE $j:= (j\pmod{6})+1$\hspace{1.02in}\{Move to the next color.\}
		\STATE \textbf{if} $\sum_{j=3}^5n_j'<n_6'$, \textbf{then} $j:=6$\hspace{.44in}\{Skip to color 6,\\ \hspace{2.48in}if few other colors remain.\}
		\STATE \textbf{if} line 3 succeeds, \textbf{then} $i:=i+1$\hspace{.20in}\{Move to $w_{i+1}$, if coloring succeeded.\}
		\STATE go to line 2
            \ENDWHILE
    \end{algorithmic}
\end{algorithm}
\end{minipage}
\bigskip

\begin{figure}[!h]
    \centering
    \begin{tikzpicture}[scale=0.8]
    \begin{scope}
    \clip (2.0,1.1) rectangle (3.6,0.1); 
    \draw[fill=gray!40!white] (3,0.6) ellipse (1cm and 0.5cm);
    \end{scope}
    \begin{scope}
    \clip (2,-0.1) rectangle (3.2,-1.1); 
    \draw[fill=gray!40!white] (3,-0.6) ellipse (1cm and 0.5cm);
    \end{scope}
    \begin{scope}
    \clip (2.0,-1.3) rectangle (2.8,-2.3); 
    \draw[fill=gray!40!white] (3,-1.8) ellipse (1cm and 0.5cm);
    \end{scope}
    \begin{scope}
    \clip (2,2.3) rectangle (4,1.3); 
    \draw[fill=gray!40!white] (3,1.8) ellipse (1cm and 0.5cm);
    \end{scope}
    \begin{scope}
    \draw[thick] (-1,0) node[uStyle, inner sep=0.7pt] (a) {\tiny{1}} (-1.35,0) node {\scriptsize{$w$}} 
    
	    (3,0.6) ellipse (1cm and 0.5cm) (3,-0.6) ellipse (1cm and 0.5cm) (3,1.8) ellipse (1cm and 0.5cm) (3,-1.8) ellipse (1cm and 0.5cm) 
	    (3,1.8) node {\footnotesize{1}} 
	    (3,0.6) node {\footnotesize{2}}
	    (3,-0.6) node {\footnotesize{3}} 
	    (3,-1.8) node {\footnotesize{4}} 
	    (1,0) ellipse (0.5cm and 2cm) (1,0) node {\scriptsize{$N(w)$}} 
	    (0.82,-1.865) -- (a.west) -- (0.82,1.865) 
	    (4.8,-0.6) node {\footnotesize{$G'$}};
    \draw[thick, decorate, decoration={calligraphic brace, amplitude=3mm}] (4.1,1.2) -- (4.1,-2.3);
    \end{scope}

    \begin{scope}[xshift=9cm]
    \clip (2,1.7) rectangle (4.5,0.7); 
    \draw[fill=gray!40!white] (3.5,1.2) ellipse (1.5cm and 0.5cm);
    \end{scope}
    \begin{scope}[xshift=9cm]
    \clip (2,0.5) rectangle (4,-0.5); 
    \draw[fill=gray!40!white] (3.5,0) ellipse (1.5cm and 0.5cm);
    \end{scope}
    \begin{scope}[xshift=9cm]
    \clip (2.0,-0.7) rectangle (3.05,-1.7); 
    \draw[fill=gray!40!white] (3.5,-1.2) ellipse (1.5cm and 0.5cm);
    \end{scope}
    \begin{scope}[xshift=9cm]
     \draw[thick] (-1,0) node[uStyle, inner sep=0.7pt] (a) {\tiny{1}} (-1.35,0) node {\scriptsize{$w$}} 
     	    (3.5,0) ellipse (1.5cm and 0.5cm) (3.5,1.2) ellipse (1.5cm and 0.5cm) (3.5,-1.2) ellipse (1.5cm and 0.5cm) 
	    (3.5,1.7) -- (3.5,0.7) (3.5,0.5) -- (3.5,-0.5) (3.5,-0.7) -- (3.5,-1.7) 
	    (2.75,1.2) node {\footnotesize{1}} 
	    (4.25,1.2) node {\footnotesize{4}} 
	    (2.75,0) node {\footnotesize{2}} 
	    (4.25,0) node {\footnotesize{5}} 
	    (2.75,-1.2) node {\footnotesize{3}} 
	    (4.25,-1.2) node {\footnotesize{6}} 
	    (5.4,1.2) node {\footnotesize{$C_1$}} 
	    (5.4,0) node {\footnotesize{$C_2$}} 
	    (5.4,-1.2) node {\footnotesize{$C_3$}}  
	    (1,0) ellipse (0.5cm and 2cm) (1,0) node {\scriptsize{$N(w)$}} 
	    (0.82,-1.865) -- (a.west) -- (0.82,1.865) 
	    (6.3,0) node {\footnotesize{$G'$}};
    \draw[thick, decorate, decoration={calligraphic brace, amplitude=3mm}] (5.6,1.7) -- (5.6,-1.7);
    \end{scope}
    \end{tikzpicture}
    \caption{The graph $G$ in the proofs of Lemmas~\ref{alg1-lem1} (left) and \ref{alg1-lem2} (right). Gray shading indicates level of fullness. For example on the left, class 1 is full, while classes 2, 3, 4 are partially full. \label{alg1-figs}}
\end{figure}
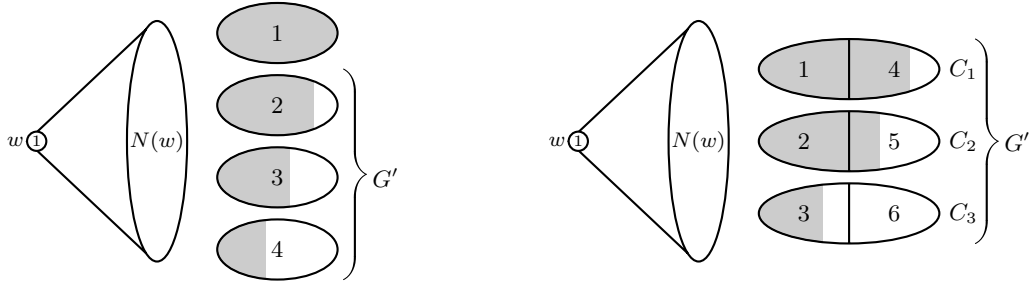

	The next $2$ lemmas use the following definition.  For a fixed vertex $w$, let $T:=N[N(w)]\setminus N[w]$. 

\begin{lem}
	If there exists $w\in V(G)$ with $2\ceil{n/6}+4\le d(w)< n/2$ and $\alpha(G[T])\le \floor{n/6}-1$, then $G$ has an equitable $6$-coloring.
	\label{alg1-lem2}
\end{lem}
\begin{proof}
	Let $G':=G\setminus N[w]$.  Since $G'$ is outerplanar, it has a $3$-coloring, say with color classes $C_1$, $C_2$, $C_3$; by \Cref{small-classes-obs},
	we assume, for all $i\in[3]$, that $|C_i|\le |G'|/2 \le (n - (n/3+4))/2 = n/3-2$.  By symmetry, we assume that $|C_1|\ge |C_2|\ge |C_3|$.
	Note that $|G'|=|G|-(d(w)+1)\ge n-((n-1)/2+1) = (n-1)/2$.  Thus, some $C_i$ has order at least $\ceil{(n-1)/6}$; so $C_1$ does.  
	We take from $C_1\cup\{w\}$ an independent set of size $n_1$ that contains $w$ and all of $T\cap C_1$; call this set $I_1$.  
	We also let $I_4':=C_1\setminus I_1$.
	Similarly, we take an independent set $I_2'$, of size at most $n_2$, from $C_2$ that contains all vertices of $C_2\cap T$.
	If $|C_2|>n_2$, then we let $I_2:=I_2'$ and we let $I_5':=C_2\setminus I_2'$; otherwise $I_5':=\emptyset$.
	Finally, we take from $C_3$ an independent set $I'_3$ of size at most $n_3$ that contains all of $C_3\cap T$.  If $|C_3|> 
	n_3$, then we let $I_3:=I'_3$ and $I_6':=C_3\setminus I_3$; see the right of \Cref{alg1-figs}.  We define $n'_j$ as in Algorithm~1. 
	Note that $n_6' = n_6 - \max\{0,|C_3|-n_3\} = \min\{n_6,n_6+n_3-|C_3|\}$.
	Similarly, $n_5' = n_5 - \max\{0,|C_2|-n_2\} = \min\{n_5,n_5+n_2-|C_2|\}$.  Since $|C_3|\le |C_2|$, this gives $n_6'\ge n_5'$.
	Likewise, $n_5'\ge n_4'$; also $n_6'\ge n_3'$ and $n_3'\ge n_2'$.  Recall that $n_1'=0$.  Thus, $n_6'\ge n_j'$ for all $j\in[5]$.

	Now we must extend each set $I_j'$ to an independent set $I_j$ of size $n_j$.  
	We again use Algorithm~1, but the analysis is more involved.  If initially $n_5'\ge n_6'-1$, then line 6 never affects the algorithm, and 
	the analysis is identical to that in the previous proof.  So assume that $n_5' <n_6'-1\le n_5$; this implies that $n_2'=0$.  If we ran Algorithm~1 
	without line 6, then we could possibly use up all other colors before using up 6.  As a result, we might produce an improper coloring, by using 
	color $6$ on successive neighbors of $w$.  But line $6$ ensures that we maintain the invariant $\sum_{j=3}^5n'_j\ge n'_6$, except for possibly 
	right before we color a vertex with $6$.  We see this as follows.  

	Let $G'':=G'-(I_1\cup I_2)$.  Note that $|G''| = |G'|-(n_1-1)-n_2 \le n - d(w) - (n_1-1)-n_2 \le n - (2n_6+4)-n_1-n_2 +1 = n_3+n_4+n_5 - n_6-3$.
	Thus, $n_3'+n_4'+n_5' = n_3+n_4+n_5 - |G''| \ge n_6 + 3$.  Hence, the desired invariant holds initially.  And line 6 ensures that the invariant
	is maintained each time through the while loop, except possibly just before coloring a vertex with $6$.  Indeed, if this invariant is violated,
	then line 6 immediately jumps to color $6$, and the invariant is restored the next time through the loop.  This ensures that the number of vertices
	that still need to be colored with $6$ never exceeds the total number that still need to be colored with $3,4,5$.  

	Finally, we must check that always $n_4'+n_5'\ge n_3''$; we define $n_j''$ as in the previous proof.  
	That is, the number of vertices that still must be colored with $4$ and $5$ is always at least the number of 
	vertices in $G'$ colored $3$ with an upcoming neighbor in $N(w)$.  This inequality holds because $n_3'+n_4'+n_5' \ge n_6' \ge n_3' + n_3''$.
	Lastly, we must verify that this final inequality still holds.
	The argument that $n_6'\ge n_3'+n_3''$ is the same as in the previous proof: the inequality holds initially because $n_6' = n_6\ge n_3 \ge n_3'+n_3''$;
	and each time that $n_6'$ decreases so does $n_3'+n_3''$.  Observe that line $6$ never breaks this invariant.
	This holds because color $3$ is the first encountered after the invariant is restored by coloring a vertex with $6$.
\end{proof}

\begin{lem}
	If there exists $w\in V(G)$ with $d(w)\ge 2\ceil{n/6}+4$ and $\alpha(G[T])\ge \floor{n/6}$, then $G$ has an equitable $6$-coloring.
\end{lem}
\begin{proof}
	First note that if there exists $x\in V(G)\setminus\{w\}$ with $d(x)\ge 2\ceil{n/6}+3$, then we are done by \Cref{2-dangerous-lem}.
	So henceforth we assume, for all $x\in V(G)\setminus\{w\}$, that $d(x)\le 2\ceil{n/6}+2$.

	Our main idea is to repeat the proof of the Partitioning Lemma, but to take as our base case a subgraph $G'$ containing all vertices at 
	distance at most $2$ from $w$, and possibly some vertices at distance $3$ from $w$, but no others.  We explicitly construct
	forest $F_1$, which contains $w$, to ensure that $\alpha_w(F_1)\ge \floor{|F_1|/3}$.  In fact, this large independent set containing $w$ has all
	of its other vertices in $T$; this is possible because, by hypothesis, $\alpha(G[T])\ge \floor{n/6}-1$.
	For each other vertex $x$ and $i\in[2]$ with $x\in F_i$, we just want to ensure that $d_{F_i}(x)\le \floor{(d_G(x)+1)/2}\le \ceil{n/6}+1$.
	Thus, we can again equitably $3$-color each $F_i$.
	
	We begin by constructing the subgraph $G'$ mentioned above.  All mentions of claims refer to the proof of the Partitioning Lemma.
	As in the proof of Claim~3, we repeatedly delete sets of $2$ or $3$ vertices 
	(via Claim~2).  We show that eventually we can reduce to a subgraph $G'$ containing $N[N[w]]$ such that each component of $G'-N[N[w]]$ is an
	isolated vertex.  Consider an arbitrary component $C$ of $G'-N[N[w]]$.  Because $G$ is maximal outerplanar, it is $2$-connected, so $C$ has $2$
	neighbors in $N[N[w]]$, say $y_1$ and $y_2$.  Since $G$ is maximal outerplanar, $y_1$ and $y_2$ are adjacent, and they have a common neighbor 
	$x$ with $x\in N(w)$.  But now we let $H:=G[V(C)\cup\{x,y_1,y_2\}]$, and apply Claim~2 to $H$ with specified edge $y_1y_2$.  If $|C|\ge 2$, then
	$|H|\ge 5$, so $H$ contains another reducible configuration avoiding both $y_1$ and $y_2$. 

	It is helpful to observe, since $G$ is maximal outerplanar, that $G[N(w)]$ is a single path.  Similar to \Cref{linear-forest}, the subgraph $G[T]$ 
	is a linear forest.  Let $G'':=G[N[w]\cup T]$, formed from $G'$ by deleting all vertices at distance $3$ from $w$; see \Cref{last-fig}.  We will construct the desired
	forest equipartition $F_1''\uplus F_2''$ of $G''$ such that the vertices in each path of $G[T]$ alternate between forests.  Thus, since
	each vertex of $V(G')\setminus V(G'')$ has either $1$ neighbor in $T$ or has $2$ successive neighbors in $T$, we can later assign these deleted 
	vertices arbitrarily to $F_1''$ and $F_2''$, calling the result $F_1'\uplus F_2'$, to maintain that $-1\le |F_2'|-|F_1'|\le 1$; we need not worry 
	about creating cycles.  We can also assume that each path of $G[T]$ has order $1$ or $3$, as follows.

	Given such a path of order at least $4$, we delete an endpoint of $P$ and its neighbor on $P$, avoiding the unique vertex of $P$ with $2$ neighbors
	in $N(w)$.  Given a forest equipartition of the resulting smaller graph, we extend it to the deleted vertices by adding them to opposite parts
	with the vertex of degree $2$ in $P$ also being in the part opposite its other neighbor on $P$.  Something similar works when $P$ has order $2$: again
	the deleted vertices are assigned to opposite parts, now with the vertex of $P$ with $2$ neighbors in $N(w)$ assigned to the part opposite that of
	its neighbor in $N(w)$ with only one neighbor on $P$.  Denote by $G'''$ the graph formed from $G''$ when each path of $G[T]$ is pared down to have
	order $3$, $1$, or $0$; see \Cref{last-fig}.  Let $T''':=V(G''')\cap T$, and analogously define $T''$.
	Note also that if $F_1'''$ contains a maximum independent set in $T'''$, then so does $F_1''$ in $T''$.

    \begin{figure}[H]
        \centering
        \begin{tikzpicture}[yscale=0.866]

        \foreach \i in {-3,-2,-1,1,3,4}
        \draw[thick] (\i-0.5,2) -- (\i,3) node[uStyle, fill=myothergrayer] {} -- (\i+0.5,2);
        
        \foreach \i in {-6,...,4}{
        \ifthenelse{\NOT\equal{\i}{0}}{\draw[thick] (0,0) node[uStyle] {} -- (\i,1);}
        } 

        \draw[thick] (0,1) node {\Large{\dots}};

        \draw[thick] (-1,1) -- (-1.5,2) -- (-2,1) (3,1) -- (3.5,2) -- (4,1) (2.5,2) -- (3,1) (4.5,2) -- (4,1) -- (5.5,2) (-5,1) -- (-5.5,2) -- (-6,1) (-5.5,2) node[uStyle] {};
        
        \foreach \i in {-4.5,-3.5,...,1.5}{
        \pgfmathparse{\i > -1.5}
        \ifthenelse{\equal{\pgfmathresult}{1}}{\draw[thick] (-1,1) -- (\i,2);}{\draw[thick] (-2,1) -- (\i,2);}
        } 

        \foreach \i in {-6,...,3}{
        \ifthenelse{\NOT\equal{\i}{-1} \AND \NOT\equal{\i}{0}}{\draw[thick] (\i,1) node[uStyle] {} -- (\i+1,1) node[uStyle] {};}
        } 

        \foreach \i in {-4.5,-3.5,...,4.5}{
        \ifthenelse{\NOT\equal{\i}{1.5}}{\draw[thick] (\i,2) node[uStyle] {} -- (\i+1,2) node[uStyle] {};}
        }

        \foreach \i in {-4.5,-3.5,0.5,1.5,2.5,3.5,4.5,5.5}
        \draw[thick] (\i,2) node[uStyle, fill=mygrayer] {};

        \draw[thick] (0,-0.3) node {$w$} (-7,1) node {$N(w)$} (-7,2) node {$T$};
        \end{tikzpicture}
        \caption{An example of $G'$.  The graph $G''$ is formed from $G'$ by removing dark gray vertices (which are at distance $3$ from $w$). 
	    The graph $G'''$ is formed from $G''$ by removing light gray vertices (which are at distance $2$ from $w$).} 
        \label{last-fig}
    \end{figure}
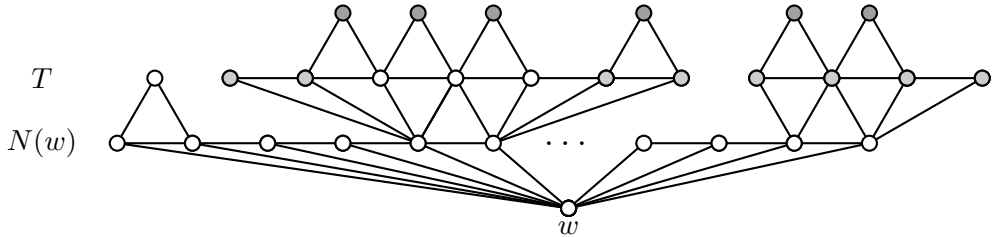

	Finally, we construct the desired forest equipartition of $G'''$.  For brevity, in the rest of the proof, we write $F_i$ to denote $F_i'''$.
	Denote by $a$ and $b$, respectively, the numbers of path components in $G[T''']$
	of orders $3$ and $1$.  We will ensure that $F_1$ contains an independent set from among $T'''$ of size $\min\{2a+b,\ceil{n/6}\}$.  If this set
	has size $\ceil{n/6}$, then we can add to it $w$ to verify that $\alpha_w(F_1)\ge \floor{n/6}$.  If the size is smaller, then we note that it is
	a maximum independent set in $G[T''']$, so also $F_1''$ contains a maximum independent set in $G[T'']$.  By assumption, the latter set has size at
	least $\floor{n/6}$.  So we will be done.

	First suppose that $3a+b\ge n/3$.  In each path of $G[T''']$ we assign the vertices alternatingly to the $2$ forests.  So from a path of order $3$, 
	one forest grows by $2$ and the other by $1$; we say that the forest growing by $2$ gets an ``extra'' vertex.  In $\floor{(a+b)/2}$ paths of 
	$G[T''']$, we assign the extra vertex to $F_1$, and in the remaining $\ceil{(a+b)/2}$ paths, we assign the extra vertex to $F_2$.
	So the size of the independent set in $F_1$ is $a(1)+b(0)+\floor{(a+b)/2} = \floor{(3a+b)/2} \ge \floor{n/6}$.  Lastly, we assign all
	vertices of $N(w)$ alternatingly to $F_2$ and $F_1$. We can easily check that this gives a forest equipartition.  We can also verify that 
	$d_{F_i}(x)\le \floor{(d_G(x)+1)/2}$ for all $x$ and $i\in[2]$ with $x\in F_i$.  So we are done.

	Now assume instead that $3a+b < n/3$.  In this case, for every path of $G[T]$, we assign the extra vertex to $F_1$, rather than to $F_2$.  
	For each path of order $3$, we mark as ``reserved'' the $2$ neighbors in $N(w)$ of its center vertex; we also assign to $F_2$ an additional 
	$a+b+1$ unreserved 
	vertices of $N(w)$.  This is possible because $a+b+1 < n/3+1 - 2a \le d(w)-2a$.  Finally, we assign alternatingly to $F_2$ and $F_1$ all 
	unassigned vertices of $N(w)$, including those that were reserved.  This ensures that $0\le |F_2|-|F_1|\le 1$ and that each $F_i$ is acyclic.
	As before, it is straightforward to check that $d_{F_i}(x)\le \floor{(d_G(x)+1)/2}$ for all $x$ and $i\in [2]$ with $x\in F_i$.
\end{proof}
\bibliographystyle{habbrv}
{\footnotesize{\bibliography{equitable}}}

\section*{Appendix}
\subsection*{Finishing the Proof of Claim~1}

Now we give the full details of the proof of Claim~1 in the proof of the Partitioning Lemma.  For convenience, we restate it below.

\begin{lem}[Claim~1 in the proof of the Partitioning Lemma]
	If $G$ is outerplanar, then there exists a supergraph $G'$ of $G$ such that 
	$G'$ is maximal outerplanar and $d_{G'}(v)\le 2\lceil n/6\rceil+3$ for all vertices $v$ but at most two, say $w_1,w_2$.  Furthermore, $d_{G'}(w_i)\le \max\{2\ceil{n/6}+4,d_G(w_i)+1\}$ for each $i\in[2]$; and if equality holds for both $d_{G'}(w_1)$ and $d_{G'}(w_2)$, then $w_1w_2\in E(G')$.
\end{lem}

\begin{proof}
	A vertex $w$ is \emph{bad} if $d(w)\ge 2\lceil n/6\rceil+3$; note that each bad vertex has degree at least 5.  Throughout, we refer to our current
	graph as $G'$.  We have $3$ phases, if needed, of adding edges.  In Phase 1, we add edges as long as possible (maintaining outerplanarity) 
	without adding any edge incident to a bad vertex.  In Phase 2 (if needed) we add a first edge incident to a bad vertex; if after Phase 1 
	$G'$ has a cut-vertex $v$, then in Phase 2 we add an edge with endpoints in distinct components of $G'-v$.  (If Phase 2 is not needed, then we
	simply say that Phase 2 is \emph{empty}; similarly, for Phase 3.) Finally, in Phase 3 (if needed) we add an edge incident to the other bad vertex.  
	By \Cref{few-big-lem}, we know that throughout $G'$ has at most $2$ bad vertices.  We will argue about the structure of $G'$ after 
	Phases 1, 2, and 3;  ultimately, we will conclude that after Phase 3, the graph $G'$ is maximal outerplanar.

    \begin{clmA}
    \label{subclaim2}
    	After Phase 1, no bad vertex in $G'$ is a cut-vertex.
    \end{clmA}
    \begin{clmproof}
   	Suppose $v$ is a cut-vertex that is bad. Let $G_1$ be a component of $G-v$ that contains no bad vertices and let $G_2:=G-(G_1\cup\{v\})$. 
	If $G_2=\{u\}$, then $u$ is not bad since bad vertices have degree at least 5. Furthermore, if $|G_2|\ge2$, then at most $1$ vertex in $G_2$ is bad. 
	In both cases, we add an edge between non-bad vertices of $G_1$ and $G_2$ while preserving outerplanarity.  
    \end{clmproof}

    \begin{figure}[h]
    \centering
    \begin{tikzpicture}[scale=1.2]
    \begin{scope}
    \draw[thick] (-1,0.5) -- (0,0) -- (1,0.5) (-1,-0.5) -- (0,0) -- (1,-0.5);
    \foreach \i/\j in {0/0, -1/0.5, 1/0.5, -1/-0.5, 1/-0.5}
    \draw[thick] (\i,\j) node[uStyle] {};
	    \draw[thick] (0,0.225) node {\footnotesize{$v$}} (-.95,0.725) node {\footnotesize{$w_1$}} (.95,0.725) node {\footnotesize{$w_2$}} (-.95,-0.725) node {\footnotesize{$x_1$}} (.95,-0.725) node {\footnotesize{$x_2$}};
	    \draw (-1,0) node {\footnotesize{$~G_1$}};
	    \draw (1,0) node {\footnotesize{$~G_2$}};
    \end{scope}

    \begin{scope}[xshift=8cm]
    \draw[thick] (-1,0.5) -- (0,0) -- (1,0.5) (-1,-0.5) -- (0,0) -- (1,-0.5) (-3,0.5) -- (-2,0) -- (-3,-0.5) (-1,0.5) -- (-2,0) -- (-1,-0.5); 
	    \draw (-1,0) node {\footnotesize{$~G_2$}};
	    \draw (-3,0) node {\footnotesize{$~G_1$}};
	    \draw (1,0) node {\footnotesize{$~G_3$}};
    \foreach \i/\j in {0/0, -1/0.5, 1/0.5, -1/-0.5, 1/-0.5, -3/0.5, -3/-0.5, -2/0}
    \draw[thick] (\i,\j) node[uStyle] {};
    \draw[thick] (0,0.25) node {\footnotesize{$v'$}} (-.95,0.725) node {\footnotesize{$w_1$}} (1,0.775) node (w2) {} (-.95,-0.725) node {\footnotesize{$x_1$}} 
	    (1,-0.8) node (x2) {} (-2,0.225) node {\footnotesize{$v$}};
    \end{scope}
    \end{tikzpicture}
	    \caption{Hypothetical cut-vertices $v$ and $v'$ in the proof of \Cref{subclaim3}.}
    \label{cut-vrt-fig}
    \end{figure}
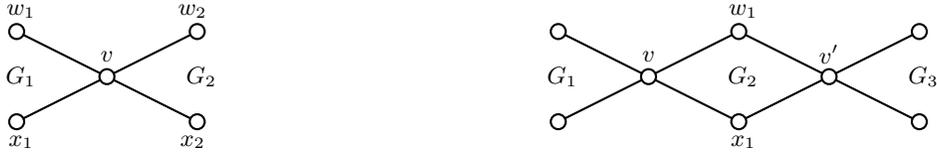

    \begin{clmA}
    \label{subclaim3}
    	If $G'$ has a cut-vertex $v$ after Phase 1, then $G'-v$ has exactly $2$ components and one of those components has bad vertices $w$ and $x$
	such that edges $vw$ and $vx$ appear on the outer facial walk.  Furthermore, such a vertex $v$ is unique; thus $G'$ is 2-connected after Phase 2.
    \end{clmA}  
	\begin{clmproof}
		We first observe that each component $G_i$ of $G'-v$ contains a neighbor $w_i$ of $v$ on the outer face that is not bad, unless $v$ has
		$2$ bad neighbors $w_i$ and $x_i$ in $G_i$ and both edges $vw_i$ and $vx_i$ appear on the boundary of the outer face.  To see this,
		first suppose that $v$ has a single neighbor $w_i$ in $G_i$.  If $|G_i|=1$, then clearly $w_i$ is not bad.  But if $|G_i|\ge 2$, then
		$w_i$ is a cut-vertex, so $w_i$ is not bad, by \Cref{subclaim2}.  Now assume instead that $v$ has at least $2$ neighbors in $G_i$.  Pick $w_i$
		and $x_i$ in $G_i$ such that edges $vw_i$ and $vx_i$ appear on the outer facial walk.  So the claimed observation holds regardless
		of whether or not both $w_i$ and $x_i$ are bad.

		Note that we can add edge $w_iw_{i+1}$ if both $w_i$ and $w_{i+1}$ are non-bad, by possibly reflecting the embedding of $G_i$ or $G_{i+1}$
		around $v$.  This proves the first statement.  If there exists a second vertex $v'$ also with the properties of $v$, then $v$ and $v'$ 
		are both adjacent along the outer face to both bad vertices.  By \Cref{subclaim2}, each bad vertex is not a cut-vertex, so each bad vertex
		has degree at most $3$, a contradiction.
	\end{clmproof}
    \begin{clmA}
    \label{subclaim4}
    	After Phase 3, the graph $G'$ is maximal outerplanar.
    \end{clmA}  
    \begin{clmproof}
		By the previous claim, after Phase 2, graph $G'$ is $2$-connected, so the boundary walk of every face in $G'$ is a cycle.  
		To prove the claim, it suffices to 
		consider an arbitrary non-outer face $f$, and show that $f$ is indeed a $3$-face.  The only reason that we might not add a chord of a 
		$4^+$-face $f$ is that one of its endpoints is bad.  But since $G'$ has at most $2$ bad vertices, if $f$ is a $5^+$-face, then we can 
		choose a chord avoiding all bad vertices, which contradicts the maximality of $G'$.  (We always embed the added chord in the interior of 
		the face.  We need not worry about the chord already existing, but being embedded outside of $f$, since this would imply that the 
		embedding is not 
		outerplanar.)  A similar argument works if $f$ is a $4$-face unless $2$ bad vertices, say $w_1$ and $w_2$, appear successively along the 
		boundary of $f$.  So assume this is the case.

		In order to achieve $2$-connectedness, the total number of edges added incident to bad vertices is at most $1$.
		And in order to triangulate all $4$-faces, the total number of edges added incident to bad vertices is at most $2$.
		But if we have added an edge of the first type, then the bad vertices $w_1,w_2$ have a common neighbor $v$, so $w_1,w_2$ lie together
		in $G''$ on at most one $4$-face.  And we ensure that it becomes triangulated when we add a single edge to finish Phase 3.
    \end{clmproof}

    This final claim completes the proof.
\end{proof}

\subsection*{Reducing the Case $\bm{s=5}$ to the Case $\bm{s=4}$ for the Main Theorem}
To conclude this appendix we show that if the Main Theorem can be extended to the case $s=4$, then it can also be extended to the case $s=5$.  This reduction is \emph{not needed} for the
proof of the Main Theorem.  But it provides modest support for \Cref{45-conj}.

\begin{lem}
\label{reduce-to-smallest-s-lem}
	If every outerplanar graph $H$ with $\min_{w\in V(H)}\alpha_w\ge \floor{|H|/4}$ has an equitable $4$-coloring, 
	then also every outerplanar graph $G$ with $\min_{x\in V(G)}\alpha_x\ge \floor{|G|/5}$ has an equitable $5$-coloring. 
\end{lem}

\begin{proof} 
Let $G$ be an outerplanar graph. 
	By the calculation at the end of \Cref{lem1}, we know that if $(n'+1)/(\alpha'_x+1)>4$, then $d_{G'}(x) > n/5+3.5$. 
	We strengthen this inequality slightly, as follows.  If $d_{G'}(w)\le 3$, then $\alpha'_w\ge 1 + (n'-d_{G'}(w)-1)/4 \ge n'/4$, as desired.
	If $d(w') \ge 4$, then we can strengthen to
	$$
	\left\lceil\frac{n'+1}{\alpha'_w+1}\right\rceil 
	\le \left\lceil\frac{n-\floor{n/5}+1}{\ceil{(n-\floor{n/5}-d(w)-1)/3}+2}\right\rceil
	\le \left\lceil\frac{n-{n/5}+1}{\ceil{(n-{n/5}-d(w)-1)/3}+2}\right\rceil.
	$$
	Note the ``$+1$'' in the final numerator rather than the ``$+2$'' (which holds for all values of $d(w)$) in the proof of \Cref{reduc-to-small-lem}.
	This final inequality holds because $(a+b+\epsilon)/((a+\epsilon)/3) \le (a+b)/(a/3)$ whenever $a,b,\epsilon > 0$.
	With this motivation, let $S:=\{x\in V(G)~|~d_G(x) \ge \ceil{n/5} + 4\}$.  Now we consider $3$ cases based on $|S|$.

\textbf{Case 1: $\bm{|S|\le 1}$.}
Let $w$ be a vertex of maximum degree, and let $I_w$ be an independent set containing $w$ of size $\floor{n/5}$.
Let $G':= G-I_w$.  Clearly $d_{G'}(x)\le n/5+4$ for all $x\in V(G)$.  So $\max_{x\in V(G')}(n'+1)/(\alpha'_x+1)\le 4$, as desired.  

\textbf{Case 2: $\bm{|S|=2}$.}
We denote $S$ by $\{w_1,w_2\}$.  We assume that no independent set $I$ of size $\floor{n/5}$ contains both $w_1$ and $w_2$;
if it does, then we let $G':=G-I$, and we are done.
For each $i\in [2]$, let $I_i$ be an arbitrary independent set containing $w_i$ with size $\floor{n/5}$, and
let $b_i:=|N(w_i)\cap I_{3-i}|$.
If $d(w_i)-b_i\le \ceil{n/5}+3$, for either $i\in [2]$, then we let $G':=G-I_{3-i}$.
Now $w_{3-i}\notin V(G')$ and $d_{G'}(w_i) = d_G(w_i)-|N(w_i)\cap I_{3-i}| = d(w_i) - b_i \le \ceil{n/5}+3$.
Thus, $\max_{x\in V(G')}(n'+1)/(\alpha'_x+1)\le 4$, as desired.  

So assume instead that $d(w_i)-b_i \ge \ceil{n/5}+4$ for each $i\in [2]$.
Let $S_i:=N(w_i)\setminus I_{3-i}$ for each $i\in [2]$.  Since $G$ has no $K_{2,3}$-minor, at most $2$ neighbors of $w_1$ appear on $w_1,w_2$-paths.
Form $S_1'$ from $S_1$ by deleting all such neighbors.  Define $S_2'$ analogously and note, for each $i\in[2]$, that $|S_i'|\ge \ceil{n/5}+2$.
Since $S_1'\subseteq N(w_1)$, we know $G[S_1']$ is $2$-colorable so contains an independent set $I_1'$ of size $|S_1'|/2\ge \ceil{n/10}+1$.
Similarly, $S_2'$ contains an independent set $I_2'$ of size $\ceil{n/10}+1$.  

	By construction $I_1'\cup I_2'$ is also independent; 
	pick $I\subseteq I_1\cup I_2$ with $|I| = \floor{(n+1)/5}$ and let $G':= G-I$.  (Note here $|I|$.  If $\floor{(n+1)/5}=\floor{n/5}$, then we
	know from the calculation above, slightly strengthening that at the end of \Cref{reduc-to-small-lem}, 
	that each other vertex of $G'$ is in an independent set that is sufficiently large.  But if $\floor{(n+1)/5} = \floor{n/5}+1$, then we repeat
	that calculation with $(n+1)/5$ in place of $n/5$.  This time we get $d(w) > n/5+21/5$, which is still good enough.)
	For each $i\in 2$, we know that $\alpha'_{w_i}\ge |I_i| = \floor{n/5} = \floor{n'/4}$.  
Thus, $\max_{x\in V(G')}(n'+1)/(\alpha'_x+1)\le 4$, so we are done.

\textbf{Case 3: $\bm{|S|\ge 3}$.}
We denote the elements of $S$ by $w_i$.  From among all $w_i$ we choose $w_1,w_2,w_3$ (possibly with renaming) to minimize the order of the 
smallest connected subgraph $J$ containing $w_1,w_2,w_3$; we will explain the motivation behind this choice near the end of Cases~3.3 and 3.4.
Let $S':=\{w_1,w_2,w_3\}$.  At each mention below of indices $i,j$ or of indices $i,j,k$, we assume that all indices are distinct.
We typically want our statements to hold for all such choices of $i,j$ or of $i,j,k$.
Our goal is to find sets $S_i\subseteq N(w_i)$, for all $i\in [3]$, such that
(1) $|S_i|\ge \lceil n/5\rceil$, (2) $S_i \cap N(w_j) = \emptyset$, and 
(3) $E(S_i,S_j)=\emptyset$.  Note that (2) implies $S_i\cap S_j = \emptyset$.
The value of such sets is that when constructing an independent set of size $\lceil n/5\rceil$ containing $w_i$, we can take $w_i$ along with 
arbitrary independent sets from $G[S_j]$ and $G[S_k]$.  Since $\alpha(G[S_j])\ge |S_j|/2 \ge \lceil n/10\rceil$ and also $\alpha(G[S_k])\ge \lceil 
n/10\rceil$, we will nearly be done.  Finally, we will need to construct big independent sets containing $w_h$ for all $h\ge 4$.

To provide more details, we consider cases based on $G[S]$.
In each case, we start by letting $S_i:=N(w_i)\setminus (N[w_j]\cup N[w_k])$ for each $i\in [3]$.
To ensure condition (3), we may also need to delete a few more vertices from certain sets $S_i$.
Throughout whenever we speak of a $w_i,w_j$-path $P$, we mean that $P$ has length at most $3$ and that $P$ is minimal under inclusion; in particular,
	$|N(w_i)\cap V(P)|=1$ and $|N(w_j)\cap V(P)|=1$.

\textbf{Case 3.1: $\bm{G[S']=K_3}$.}
We focus on ensuring that $E(S_i,S_j)=\emptyset$ in the case $i=1$ and $j=2$, but the same argument works for each choice of distinct $i,j\in[3]$.
Recall that $G$ has no $K_{2,3}$-minor.  Since $w_3\in N(w_1)\cap N(w_2)$, the maximum number of $w_1,w_2$-paths in $G-w_3$ is at most $1$.
(Here the $w_1,w_2$-paths need not be disjoint, just distinct.  If they intersect, we contract this intersection onto its neighbor in 
$\{w_1,w_2\}$.  And we still find a $K_{2,3}$-minor.)  So if such a $w_1,w_2$-path exists, then we simply remove one of its interior
vertices from $S_1$.  We also do the same for $w_1,w_3$-paths.  Thus, in total we remove at most $2$ vertices from our original choice of $S_1$.
So, in the end we have $|S_1|\ge d(w_1)-4 \ge \lceil n/5\rceil$, as desired.  Similarly, we get $|S_2|\ge \lceil n/5\rceil$ and 
$|S_3|\ge \lceil n/5\rceil$.

Since $G$ is outerplanar, $G[S_i]$ is a disjoint union of paths.  We can $2$-color these so that each color class has size at least 
$\lfloor \lceil n/5\rceil/2\rfloor$.  Thus, to construct an independent set $I$ of size $\lfloor n/5\rfloor$ containing $w_1$, we take $w_1$
along with the smaller color class in the $2$-coloring of $G[S_2]$ and the smaller class in the $2$-coloring of $G[S_3]$, 
possibly discarding a vertex to get to size precisely $\lfloor n/5\rfloor$.  Let $G':=G-I$ and $n':=|V(G')|$.  To get an independent set $I'$
in $G'$ containing $w_2$ and having size $\lfloor n'/4\rfloor\le \lceil n/5\rceil$, we take $w_2$ along with the bigger color class in the 
$2$-coloring of $G[S_1]$ and the remaining color class in the $2$-coloring of $G[S_3]$.  
In fact, this argument shows that $\alpha'_{w_2} \ge 2\lceil(\lceil n/5\rceil/2)\rceil + 1 \ge \lceil n/5\rceil+1$.
The same argument works to get an independent set in $G'$ containing $w_3$.

Finally, we consider $w_h$ with $h\ge 4$.  We start with the union of the bigger color class in each $S_i$.  Vertex $w_h$ could have neighbors in
this union, but at most $2$; otherwise we contract $S'$ to a single vertex and get $K_{2,3}$.  So we add $w_h$ to this union and remove its at 
most $2$ neighbors.  The resulting independent set has size at least $3\ceil{n/10}+1-2 \ge \ceil{n/5} \ge \floor{n'/4}$.

\textbf{Case 3.2: $\bm{G[S']=P_3}$.}
By symmetry, we assume that $w_2,w_3\in N(w_1)$.
We can no longer guarantee the sets $S_1,S_2,S_3$ as above, with $|S_i|\ge n/5$ for all $i\in [3]$.
In particular, $w_1$ could have $2$ common neighbors with one or both of $w_2$ and $w_3$.
First suppose that $w_1$ has at most one common neighbor with each of $w_2$ and $w_3$.  Note that $|S_1|\ge d(w_1)-4$.  
For each $w_1,w_2$-path $P$, remove from $S_2$ the neighbor of $w_2$ on $P$ (but remove nothing from $S_1$).
And for each $w_1,w_3$-path $P$, remove from $S_3$ the neighbor of $w_3$ on $P$ (but remove nothing from $S_1$).
As above, it is easy to check that $|S_i|\ge \ceil{n/5}$ for all $i$, so we can finish analogously.

Now assume instead, by symmetry, that $w_1$ has $2$ common neighbors with $w_2$.  In this case, we will ensure that
$|S_1|\ge \ceil{n/5}-2$, but that $|S_2|\ge \ceil{n/5}+1$ and $|S_3|\ge \ceil{n/5}+1$.  Because $w_2$ has $2$ common neighbors with $w_1$, there
do not exist any $w_2,w_3$-paths in $G-N[w_1]$, or we get a $K_{2,3}$-minor.  
	So, $|S_2| = d(w_2) - |N(w_2)\cap N[w_1]| \ge (\ceil{n/5} + 4) - 3 = \ceil{n/5}+1$.
	For each $w_1,w_3$-path $P$, we remove from $S_3$ the neighbor of $w_3$ on $P$.  So similarly, we have $|S_3| \ge d(v) - 3 \ge \ceil{n/5} + 1$.
	And recall that $|S_1| = d(w_1) - |N(w_1)\cap N[w_2]| - |N(w_1)\cap N[w_3]| \ge \ceil{n/5}+4 - 2(3) = \ceil{n/5}-2$.

Now for our independent set $I$ containing $w_1$, we take the smaller color class in each of $G[S_2]$ and $G[S_3]$, along with $w_1$.
This set has size at least $2\lfloor (\lceil n/5\rceil +1)/2\rfloor + 1\ge \lceil n/5\rceil$.  As usual, we let $G':=G-I$.
For an independent set $I'$ in $G'$ of size $\floor{n'/4} = \floor{(n+1)/5}$ containing $w_2$, we take 
$w_2$ along with the bigger color class in each of $G[S_1]$ and $G[S_3]$;
this set has size at least $\lceil(\lceil n/5\rceil + 1)/2\rceil + \lceil(\lceil n/5\rceil - 2)/2\rceil + 1 = 
\lceil(\lceil n/5\rceil + 1)/2\rceil + \lceil\lceil n/5\rceil/2\rceil = \lceil n/5\rceil + 1$.
The same argument works for $w_3$.
Finally, we consider $w_h$ with $h\ge 4$. Similar to above, we start with the bigger color class in each $G[S_i]$.  This set has at most $2$ neighbors
of $w_h$, so we remove them and add $w_h$.  
The resulting set has size at least $\ceil{|S_1|/2}+\ceil{|S_2|/2}+\ceil{|S_3|/2}-2+1\ge (\ceil{n/5}+1)+(\ceil{n/10}-1)-2+1 = \ceil{n/5}+\ceil{n/10}-1 \ge \ceil{n/5}\ge \floor{n'/4}$, as desired.

\textbf{Case 3.3: $\bm{G[S']=\overline{P_3}}$.}
By symmetry, assume that $w_2w_3\in E(G)$ and $w_1w_2,w_1w_3\notin E(G)$.  Note that at most $2$ neighbors of $w_1$ lie on paths from $w_1$ to 
$\{w_2,w_3\}$; otherwise, $G$ has a $K_{2,3}$-minor.  Thus, by removing all such neighbors from $S_1$, we can ensure that $|S_1|\ge d(w_1)-2\ge \ceil{n/5}+2$.
Likewise, since $w_1\notin N(w_2)\cup N(w_3)$, even after removing all neighbors on paths to the other vertices of $S'$, we can ensure that 
$|S_2|\ge \ceil{n/5}-1$ and $|S_3|\ge \ceil{n/5}-1$.  
But actually, we cannot have both of these bounds hold with equality.  
If each of $w_2$ and $w_3$ has both $2$ neighbors on paths to $w_1$ and $2$ neighbors on paths to the other of $w_2,w_3$,
then $G$ contains a $K_{2,3}$-minor.  Thus, we assume by symmetry that $|S_2|\ge \ceil{n/5}$ and $|S_3|\ge \ceil{n/5}-1$.
We construct an independent set $I$ consisting of the smaller color class in each of $G[S_2]$ and $G[S_3]$, along with $w_1$.
This set has size at least $\floor{(\ceil{n/5} -1)/2} + \floor{\ceil{n/5}/2} + 1 = (\ceil{n/5}-1)+1 = \ceil{n/5}$.  
Let $G':= G-I$.
In $G'$, we construct an independent set $I'$ consisting of the bigger color class in $G[S_1]$, the bigger color class in $G[S_3]$, 
along with $w_2$.  This set has size at least 
$\ceil{(\ceil{n/5} + 2)/2} + \ceil{(\ceil{n/5}-1)/2} + 1=
\ceil{\ceil{n/5}/2} + \ceil{(\ceil{n/5} +1)/2} + 1=
\ceil{n/5}+ 2 \ge \floor{n'/4}$. 
And the independent set for $w_3$ is similar.

Finally, we consider $w_h$ with $h\ge 4$.  Now we start with the independent set above in $G'$ containing $w_2$, remove $w_2$, add $w_h$, add the bigger color class in $G[S_2]$, and remove all neighbors of $w_h$.  Before removing neighbors of $w_h$, this set has size at least $(\ceil{n/5}+2)+\ceil{\ceil{n/5}/2}$. 
Clearly, $\ceil{\ceil{n/5}/2}\ge 1$.
So it suffices to show that $w_h$ has at most $3$ neighbors in this set; suppose not.  We now contract the smallest connected subgraph $J$ containing
all of $S'$ to a single vertex, and get a copy of $K_{2,3}$.  Note that $w_h\notin N(w_2)\cup N(w_3)$, by our choice of $w_1,w_2,w_3$ to minimize the 
order of $J$.  First suppose that $w_h$ has a neighbor in $S_2\cup S_3$.  This implies that $w_1$ must have a neighbor in $N(w_2)\cup N(w_3)$.
So $|V(J)|=4$.  Thus, at most one neighbor of $w_h$ in $S_1\cup S_2\cup S_3$ is contracted away; so $3$ remaining neighbors of $w_h$ give the part of 
size $3$ in the copy of $K_{2,3}$.  So assume instead that $w_h$ has no neighbor in $S_2\cup S_3$.  Thus, all neighbors of $w_h$ in $S_1\cup S_2\cup S_3$
lie in $S_1$.  But this must be at most $2$ such neighbors, since $G$ has no $K_{2,3}$-minor, a contradiction.

\textbf{Case 3.4: $\bm{G[S']=\overline{K_3}}$.}
This case is similar to the previous ones, but we need to prove stronger lower bounds on $|S_i|$ to allow that $w_h$ might have up to $4$ neighbors in
$S_1\cup S_2\cup S_3$.  
If, for every pair $w_i,w_j$ the subgraph $G-N[w_k]$ contains a $w_i,w_j$ path, then each $N(w_i)$ contains at most $2$ vertices on $w_i,\{w_j,w_k\}$-paths.
We remove all of these and get $|S_i|\ge \ceil{n/5}+2$ for all $i$.  So assume instead, by symmetry, that $G$ contains no such $w_1,w_2$-path. 
If $G$ contains no $w_1,w_2$-path of length at most $3$, then even after removing all neighbors of $w_i$ on paths to $\{w_j,w_k\}$, for each choice of $i,j,k$,
we have $|S_1|\ge \ceil{n/5}+2$, $|S_2|\ge \ceil{n/5}+2$, and $|S_3|\ge \ceil{n/5}$.  So assume instead that such a path $P$ exists.  If $P$ has length $2$,
then its interior vertex $x$ lies in $N(w_1)\cap N(w_2)\cap N(w_3)$.  Furthermore, every $w_1,w_2$-path of length at most $3$ contains $x$.  Now removing
from $S_i$ all neighbors of $w_i$ on $w_i,w_3$-paths, for each $i\in [2]$, we still have 
$|S_1|\ge \ceil{n/5}+2$, $|S_2|\ge \ceil{n/5}+2$, and $|S_3|\ge \ceil{n/5}$.
So we assume below that $N(w_1)\cap N(w_2)\cap N(w_3) = \emptyset$.

So assume instead that (no such path $w_1,w_2$-path of length $2$ exists and) $P$ has length $3$.
Denote $P$ by $w_1x_1x_2w_2$.  By symmetry, assume that $x_2\in N(w_3)$.  So $w_1$ has at most $2$ neighbors on $w_1,\{w_2,w_3\}$-paths.
If $G$ also contains a $w_1,w_3$-path disjoint from $N(w_2)$, then also $w_2$ has at most $2$ neighbors on paths to $\{w_1,w_3\}$,
so $|S_1|\ge \ceil{n/5}+2$, $|S_2|\ge \ceil{n/5}+2$, and $|S_3|\ge \ceil{n/5}$.  So assume that $G$ contains no such path.
We will show that every $w_3,w_1$-path contains $x_2$.  Suppose not.  Consider a $w_3,w_1$-path $P$ of length $3$, and let $x_3$
be an internal vertex of $P$ such that $x_3\ne x_2$ and $x_3\in N(w_2)$.  Now $G$ has a $K_{2,3}$-minor with $\{x_2,x_3\}$ as one part and 
with $\{w_1,w_2,w_3\}$ as the other part.  Thus, in every case we get $|S_1|\ge \ceil{n/5}+2$, $|S_2|\ge \ceil{n/5}+2$, and $|S_3|\ge \ceil{n/5}$.
For our independent set $I$, we take $w_1$, along with the smaller color class in $G[S_2]$ and the smaller color class in $G[S_3]$.
This gives an independent set of size bigger than $\ceil{n/5}$, and we take $I$ to be an arbitrary subset of size $\floor{n/5}$.  
Let $G':=G-I$.  The analysis is straightforward for independent sets in $G'$ of size $\floor{n'/4}\le \ceil{n/5}$
that contain $w_2$ or contain $w_3$.

Finally, we consider $w_h$ with $h\ge 4$.  We start with the independent set containing the bigger color class in each of $G[S_1]$, $G[S_2]$, and $G[S_3]$.
Now we add $w_h$ and remove all neighbors of $w_h$.  The size of this set before removing neighbors of $w_h$ is at least
$2\ceil{(\ceil{n/5}+2)/2}+\ceil{\ceil{n/5}/2}+1\ge \ceil{n/5}+\ceil{n/10}+3\ge \ceil{n/5}+4$.  So it suffices to show that $w_h$ has at most $4$ neighbors
in the set.  If $w_h$ has more than $4$ neighbors in $S_1\cup S_2\cup S_3$, then it has neighbors in each $S_i$ (since it has at most $2$ in each).
This implies that $|V(J)|\le5$; otherwise, we could get a smaller $J$ by replacing some $w_i$ with $w_h$.  Now we contract $J$ to a single vertex.
Doing so contracts away at most $2$ neighbors of $w_h$.  Thus, the remaining at least $3$ neighbors of $w_h$ in $S_1\cup S_2\cup S_3$ give rise to a 
copy of $K_{2,3}$, a contradiction.
\end{proof}

\end{document}